\documentclass[11pt]{amsart}

\usepackage{graphicx}
\usepackage{amsmath}
\usepackage{amssymb}
\usepackage{amsopn}
\usepackage{mathrsfs}
\usepackage{enumerate}
\usepackage{txfonts}

\DeclareMathOperator*{\Vol}{Vol}
\DeclareMathOperator*{\ddim}{dim}

\DeclareMathOperator{\vdm}{{Vdm}}
\DeclareMathOperator*{\support}{supp}
\DeclareMathOperator*{\Card}{Card}
\DeclareMathOperator{\interior}{int}

\DeclareMathOperator*{\capa}{cap}
\DeclareMathOperator*{\ddc}{dd^c}
\DeclareMathOperator*{\Span}{span}

\DeclareMathOperator*{\adj}{adj}

\newcommand{\gl}{\prec}

\renewcommand{\wp}{\mathscr{P}}

\newcommand{\C}{\mathbb{C}}
\newcommand{\R}{\mathbb{R}}
\newcommand{\N}{\mathbb{N}}

\newcommand{\bs}[1]{\boldsymbol{#1}}
\newcommand{\z}{\zeta}
\newcommand{\psh}{\mathcal{PSH}}
\newcommand{\ddcn}[1]{\left(\ddc #1\right)^n}

\newtheorem{theorem}{Theorem}[section]

\newtheorem{proposition}{Proposition}[section]
\newtheorem{lemma}{Lemma}[section]

\newtheorem{remark}[theorem]{Remark}

\title{Pluripotential Numerics}
\date{\today}
\author{Federico Piazzon}
\address{Department of Mathematics \emph{Tullio Levi-Civita}, Universit\'a di Padova, Italy.}
\email{\underline{fpiazzon@math.unipd.it}} 
\urladdr{http://www.math.unipd.it/~fpiazzon/}

\begin{document}

\begin{abstract}
We introduce numerical methods for the approximation of the main (global) quantities in Pluripotential Theory as the \emph{extremal plurisubharmonic function} $V_E^*$ of a compact $\mathcal L$-regular set $E\subset \C^n$, its \emph{transfinite diameter} $\delta(E),$ and the \emph{pluripotential equilibrium measure} $\mu_E:=\ddcn{V_E^*}.$  

The methods rely on the computation of a \emph{polynomial mesh} for $E$ and numerical orthonormalization of a suitable basis of polynomials. We prove the convergence of the approximation of $\delta(E)$ and the uniform convergence of our approximation to $V_E^*$ on all $\C^n;$ the convergence of the proposed approximation to $\mu_E$ follows. Our algorithms are based on the properties of polynomial meshes and Bernstein Markov measures.

Numerical tests are presented for some simple cases with $E\subset \R^2$ to illustrate the performances of the proposed methods.   
\end{abstract}
\maketitle
\tableofcontents

\section{Introduction}

Let $E\subset \C$ be a compact infinite set. Polynomial interpolation of holomorphic functions on $E$ and its asymptotic are intimately related with logarithmic potential theory, i.e., the study of subharmonic function of logarithmic growth on $\C$. This is a well established classical topic whose study goes back to Bernstein, Fekete, Leja, Szeg\"o, Walsh and many others; we refer the reader to \cite{Rans}, \cite{WAL} and \cite{SaTo97} for an extensive treatment of the subject.

To study polynomial interpolation on a given compact set $E$ one introduces the Vandermonde determinant (usually with respect to the monomial basis) and, for any degree $k\in \N,$ tries to maximize its modulus among the tuples of $k+1$ distinct points. Any such array of points is termed Fekte array of order $k$ for $E.$ A primary interest on Fekete arrays is that one immediately has the bound
$$\Lambda_k(z_0,\dots,z_k):=\sup_{z\in E}\sup_{f\in \mathscr C(E),f\neq 0}\frac{|I_k[f](z)|}{\|f\|_{E}}\leq (k+1)$$ 
for the Lebesgue constant $\Lambda_k$ (i.e., the norm of the polynomial interpolation operator $I_k$) for any Fekete array of order $k$ for $E.$

On the other hand Fekete arrays provide the link of polynomial interpolation with potential theory. Indeed, the logarithmic energy $\capa(E)$ of a unit charge distribution on $E$ at equilibrium
$$\capa(E):=\exp\left(-\min_{\mu\in \mathcal M^1(E)}\int_E\int_E \log\frac 1{|z-\z|}d\mu(z)d\mu(\z)   \right)$$
turns out to coincide with certain asymptotic of the modulus of the Vandermonde determinant computed at any sequence of Fekete points and with respect to the monomial basis.  Since the considered asymptotic is a geometric mean of mutual distances of Fekete points, it is termed \emph{transfinite diameter} of $E$ and denoted by $\delta(E),$
$$\delta(E):=\lim_k\left|\vdm(z_0,\dots,z_k)  \right|^{\frac{1}{\ddim \wp^k}},\;\;\;\text{ for }(z_0,\dots,z_k)\text{ Fekete array}.$$
 
The Fundamental Theorem of Logarithmic Potential Theory asserts that $\delta(E)=\capa(E)=\tau(E)$, where $\tau(E)$ is the \emph{Chebyshev constant} of $E$ and is defined by means of asymptotic of certain normalized monic polynomials. Moreover, provided $\delta(E)\neq 0$, for any sequence of arrays having the same asymptotic of the Vandermonde determinants as Fekete points the sequence of uniform probability measures supported on such arrays converges weak star to the unique minimizer of the logarithmic energy minimization problem, that is the \emph{equilibrium measure} $\mu_E$ of $E$.

The other fundamental object in this theory is the \emph{Green function} with pole at infinity $g_E(\cdot,\infty)$ for the domain $\C\setminus \hat E$, where $\hat E$ is the polynomial hull of $E.$
\begin{equation}\label{greenfunctiondefinition}
g_E(z,\infty):=\limsup_{\z\to z}\left(\sup\left\{u(\z):u\in \mathcal L(\C), u|_E\leq 0 \right\}\right).
\end{equation}
Here $\mathcal L(\C)$ is the Lelong class of subharmonic function of logarithmic growth, i.e., $u(z)-\log|z|$ is bounded near infinity.
 
It turns out that, provided $\delta(E)\neq 0$, one has
\begin{equation}\label{greenfunctiondeequivalentfinition}
g_E(z,\infty)=\limsup_{\z\to z}\left(\sup\left\{\frac 1{\deg p}\log|p(\z)|,\, p\in \wp,\,\|p\|_E\leq 1\right\}\right).
\end{equation}

It follows that the Green function of $E$ is intimately related to polynomial inequalities and polynomial interpolation: for instance, one has the Bernstein Walsh Inequality
\begin{equation}\label{BernsteinWalshUnivariate}
|p(z)|\leq \|p\|_E\exp(\deg p\, g_E(z,\infty)),\;\forall p\in \wp
\end{equation}
and the Bernstein Walsh Theorem \cite{WAL}, that relates the rate of decrease of the error of best polynomial uniform approximation on $E$ of a function $f\in hol(\interior E)\cap \mathscr C(E)$ to the possibility of extending $f$ holomorphically to a domain of the form $\{g_E(z,\infty)<c\}.$  

Lastly, it is worth to recall that, using the fact that $\log|\cdot|$ is the fundamental solution of the Laplace operator in $\C$, it is possible to show that
$$\Delta g_E(z,\infty)=\mu_E$$  
in the sense of distributions.

When we move from the complex plane to the case of $E\subset \C^n$, $n>1$, the situation becomes much more complicated. Indeed, one can still define Fekete points and look for asymptotic of Vandermonde determinants with respect to the graded lexicographically ordered monomial basis computed at these points, but this is no more related to a linear convex functional on the space of probability measures (the logarithmic energy) neither to a linear partial differential operator (the Laplacian) as in the case of $\C.$

During the last two decades (see for instance \cite{Kli}, \cite{levnotes}), a \emph{non linear} potential theory in $\C^n$ has been developed: \emph{Pluripotential Theory} is the study of \emph{plurisubharmonic functions}, i.e., functions that are upper semicontinuous and subharmonic along each complex line. Plurisubharmonic functions in this setting enjoy the role of subharmonic functions in $\C$, while \emph{maximal plurisubharmonic functions} replace harmonic ones; the geometric-analytical relation among the classes of functions being the same. It turned out that this theory is related to several branches of Complex Analysis and Approximation Theory, exactly as happens for Logarithmic Potential Theory in $\C$. 

It was first conjectured by Leja that Vandermonde determinants computed at Fekete points should still have a limit and that the associated probability measures sequences should converge to some unique limit measure, even in the case $n>1.$ The existence of the asymptotic of Vandermonde determinants and its equivalence with a multidimensional analogue of the Chebyshev constant was proved by Zaharjuta \cite{Za74I}, \cite{Za74II}. Finally the relation of Fekete points asymptotic with the equilibrium measure and the transfinite diameter in Pluripotential Theory (even in a much more general setting) has been explained by Berman Boucksom and Nymstr\"om very recently in a series of papers; \cite{BeBoNy11}, \cite{BeBo10}. Indeed, the situation in the several complex variables setting is very close to the one of logarithmic potential theory, provided a suitable translation of the definitions, though the proof and the theory itself is much much more complicated.

Since Logarithmic Potential Theory has plenty of applications in Analysis, Approximation Theory and Physics, many numerical methods for computing approximations to Greens function, transfinite diameters and equilibrium measure has been developed following different approaches as Riemann Hilbert problem \cite{Ol11}, numerical conformal mapping \cite{EmTre99}, linear or quadratic optimization \cite{Ro97,RaRo07} and iterated function systems \cite{Ma13}.

On the other hand, to the best authors' knowledge, there are no algorithms for approximating the corresponding quantities in Pluripotential Theory; the aim of this paper is to start such a study. This is motivated by the growing interest that Pluripotential Theory is achieving in applications during the last years. We mention, among the others, the quest for nearly optimal sampling in least squares regression \cite{NaJaZh16,ShXi16,CoMi16}, random polynomials \cite{BlLe15,ZeZe10} and estimation of approximation numbers of a given function \cite{Tr17}.

Our approach, first presented in the doctoral dissertation \cite[Part II Ch. 6]{Pi16T}, is based on certain sequences, first introduced by Calvi and Levenberg \cite{CL08}, of finite subsets of a given compact set termed \emph{admissible polynomial meshes} having slow increasing cardinality and for which a sampling inequality for polynomials holds true. The core idea of the present paper is  inspired by the analogy of sequences of uniform probability measures supported on an admissible mesh with the class of Bernstein Markov measures (see for instance \cite{Pi16T}, \cite{BlLePiWi15} and \cite{B97}).

Indeed, we use $L^2$ methods with respect to these sequences of measures: we can prove rigorously that our $L^2$ maximixation procedure leads to the same asymptotic as the $L^\infty$ maximization that appears in the definitions of the transfinite diameter (or other objects in Pluripotential Theory), this is due to the sampling property of admissible meshes.     
On the other hand the slow increasing cardinality of the admissible meshes guarantees that the complexity of the computations is not growing too fast.

We \emph{warn the reader} that, though all proposed examples and tests are for \emph{real sets} $E\subset \R^n\subset \C^n$, our results hold in the general case $E\subset \C^n.$ This choice has been made essentially for two reasons: first, the main examples for which we have analytical expression to compare our computation with are real, second, the case of $E\subset \R^2$ is both computationally less expensive and easier from the point of view of representing the obtained results. 

The methods we are introducing in the present work are suitable to be extended in at least three directions. First, one may consider \emph{weighted pluripotential theory} (see for instance \cite{BlLe03}) instead of the classical one: the theoretical results we prove here can be recovered in such a more general setting by some modifications. It is worth to mention that a relevant part \emph{of the proofs} of our results rest upon this weighted theory even if is not presented in such a framework, since we extensively use the results of the seminal paper \cite{BeBoNy11}. However, in order to produce the same algorithms in the weighted framework,   one needs to work with weighted polynomials, i.e., functions of the type $p(z) w^{\deg p}(z)$ for a given weight function $w$ and typically $E$ needs no more to be compact in this theory, these changes cause some theoretical difficulties in constructing suitable admissible meshes and may carry non trivial numerical issues as well.

Second, we recall that pluripotential theory has been developed on certain ''lower dimensional sets'' of $\C^n$ as sub-manifolds and affine algebraic varieties. If $E$ is a compact subset of an algebraic subset $\mathcal A$ of $\C^n$, then one can extend the pluripotential theory of the set $\mathcal A_{reg}$ of regular points of $\mathcal A$ to the whole variety and use traces of \emph{global} polynomials in $\C^n$ to recover the extremal plurisubharmonic function $V_E^*(z,\mathcal A);$ see \cite{Sa82}. This last direction is probably even more attractive, due to the recent development of the theory itself especially when $E$ lies in the set of real points of $\mathcal A,$ see for instance \cite{Ma11} and \cite{BeOr15}.

Lastly, we mention an application of our methods that is ready at hand. Very recently polynomial spaces with non-standard degree ordering (e.g., not total degree nor tensor degree) start to attract a certain attention in the framework of random sampling \cite{CoMi16}, Approximation Theory \cite{Tr17}, and Pluripotential Theory \cite{BoLe17}. For instance, one can consider spaces of polynomials of the form $\wp^k_q:=\Span \{ z^{\alpha}, \alpha_i\in \N^n,q(\alpha)<k\}$, where $q$ is any norm or even, more generally, $\wp^k_P:=\Span \{ z^{\alpha}, \alpha_i\in \N^n, \alpha \in k P\}$ for any $P\subset\R_+^n$ closed and star-shaped with respect to $0$ such that $\cup_{k\in \N}k P=\R_+^n.$  Since this spaces are being used only very recently, many theoretical questions, from the pluripotential theory point of view, arise.  Our methods can be used to investigate conjectures in this framework by a very minor modification of our algorithms.  

The paper is structured as follows. In \textbf{Section \ref{SectionPreliminaries}} we introduce \emph{admissible meshes} and  all the definitions and the tools we need from \emph{Pluripotential Theory}.

Then we present our algorithms of approximation: for each of them we prove the convergence and we illustrate their implementation and their performances by some numerical tests; we stress that, despite the fact that we will consider only cases of $E\subset \R^2$ for relevance and simplicity, our techniques are fine for general $E\subset \C^n.$ We consider the \emph{extremal function} $V_E^*$ (Pluripotential Theory counterpart of the Green function, see \eqref{SectionExtremalfunction} below) in \textbf{Section \ref{SectionExtremalfunction}}, the \emph{transfinite diameter} $\delta(E)$ in \textbf{Section \ref{SectionTransfinite}} and the \emph{pluripotential equilibrium measure} $\mu_E:=\ddcn{V_E^*}$ in \textbf{Section \ref{SectionEquilibrium}}.

All the experiment are performed with the MATLAB software \textbf{PPN package}, see\\ http://www.math.unipd.it/$\sim$fpiazzon/software.

\section{Preliminaries}\label{SectionPreliminaries}
\subsection{Pluripotential theory: some definitions}
Let $\Omega\subset \C$ be any domain and $u:\Omega\rightarrow\R\cup\{-\infty\}$, $u$ is said to be \emph{subharmonic} if $u$ is upper semicontinuous and $u(z)\leq \frac 1{2\pi r}\int_{|z-\z|=r}u(\z)ds(\z)$ for any $r>0$ and $z\in \Omega$ such that $B(z,r)\subseteq \Omega.$ 

A function $u:\Omega\rightarrow\R\cup\{-\infty\}$, where $\Omega\subset \C^n$, is said to be \emph{plurisubharmonic} if $u$ is upper semicontinuous and is subharmonic along each complex line (i.e., each complex one dimensional affine variety); the class of such functions is usually denoted by $\psh(\Omega).$ It is worth to stress that the class of plurisubharmonic functions is strictly smaller than the class of subharmonic function on $\Omega$ as a domain in $\R^{2n}.$

The \emph{Lelong class} of plurisubharmonic function with logarithmic growth at infinity is denoted by $\mathcal L(\C^n)$ and $u\in \mathcal L(\C^n)$ iff $u\in \psh(\C^n)$ is a locally bounded function such that $u(z)-\log|z|$ is bounded near infinity.

Let $E\subset \C^n$ be a compact set. The \emph{extremal function} $V_E^*$ (also termed pluricomplex Green function) of $E$ is defined mimicking one of the possible definitions in $\C$ of the Green function with pole at infinity; see \eqref{greenfunctiondefinition}.
\begin{align}\
V_E(\z):=&\sup\{u(\z)\in \mathcal L(\C^n), u|_E\leq 0\},\\
V_E^*(z):=&\limsup_{\z\to z}V_E(\z).\label{EquationExtremaDef}
\end{align}
It turns out that the extremal function enjoys the same relation with polynomials of the Green function; precisely Siciak introduced \cite{Si81}
\begin{align}
\tilde V_E(\z):=&\sup\left\{\frac 1{\deg p}\log|p(\z)|,p\in \wp, \|p\|_E\leq 1\right\},\\
\tilde V_E^*(z):=&\limsup_{\z\to z}V_E(\z).\label{EquationSiciakDef}
\end{align} 
and shown (the general statement has been proved by Zaharjuta) that $\tilde V_E^*\equiv V_E^*$ and $\tilde V_E=_{\text{q.e.}} V_E.$ Here \emph{q.e.} stands for \emph{quasi everywhere} and means for each $z\in \C^n\setminus P$ where $P$ is a \emph{pluripolar set}, i.e., a subset of the $\{-\infty\}$ level set of a plurisubharmonic function not identically $-\infty$. In the case that $V_E^*$ is also continuous, the set $E$ is termed $\mathcal L$-regular.

A remarkable consequence of this equivalence is that the Bernstein Walsh Inequality \eqref{BernsteinWalshUnivariate} and Theorem (see \cite{Si81}) hold even in the several complex variable setting simply replacing $g_E(z,\infty)$ by $V_E^*(z)$; see for instance \cite{levnotes}.

In Pluripotential Theory the role of the Laplace operator is played by the \emph{complex Monge Ampere operator} $\ddcn{}.$ Here $\ddc=2i\bar\partial\partial$, where $\partial u(z):=\sum_{j=1}^n\frac{\partial u(z)}{\partial z_j}dz_j$, $\bar\partial u(z):=\sum_{j=1}^n\frac{\partial u(z)}{\partial \bar z_j}d\bar z_j$ and $\ddcn u=\ddc u\wedge\ddc u\dots\wedge \ddc u.$ For $u\in \mathscr C^2(\C^n,\R)$ one has $\ddcn{u}=c_n \det[\partial^2u/\partial z_i\partial\bar z_j] d \text{Vol}_{\C^n}.$ The Monge Ampere operator extends to locally bounded plurisubharmonic functions as shown by Bedford and Taylor \cite{BeTa82}, \cite{BeTa76}, $\ddcn{u}$ being a positive Borel measure. 

The equation $\ddcn u=0$ on a open set $\Omega$ (in the sense of currents) characterize the \emph{maximal plurisubharmonic functions}; recall that $u$ is maximal if for any open set $\Omega'\subset\subset \Omega$ and any $v\in \psh(\Omega)$ such that $v|_{\partial \Omega'}\leq u|_{\partial \Omega'}$ we have $v(z)\leq u(z)$ for any $z\in \Omega'.$ 

Given a compact set $E\subset \C^n,$ two situations may occur: either $V_E^*\equiv +\infty$, or $V_E^*$ is a locally bounded plurisubharmonic function. The first case is when $E$ is pluripolar, it is \emph{too small for pluripotential theory}. In the latter case $\ddcn{V_E^*}=0$ in $\C^n\setminus E$ (i.e., $V_E^*$ is maximal on such a set), in other words the positive measure  $\ddcn{V_E^*}$ is supported on $E$. Such a measure is usually denoted by $\mu_E$ and termed the \emph{(pluripotential) equilibrium measure} of $E$ by analogy with the one dimensional case.

Let us introduce the \emph{graded lexicographical strict order} $\gl$ on $\N^n.$ For any $\alpha,\beta\in \N^n$ we have 
\begin{align}
\alpha\gl\beta\;\text{ if }&|\alpha|<|\beta|\text{ or  }
\begin{cases}
|\alpha|= |\beta|& \alpha\neq \beta \\
\alpha_{\bar j}<\beta_{\bar j}
\end{cases},\\
\text{ where }&\bar j:=\min \{j\in \{1,2,\dots,n\}: \alpha_j\neq \beta_j\}.
\end{align}
This is clearly a total (strict) well-order on $\N^n$ and thus it induces a bijective map
$$\alpha:\N\longrightarrow\N^n.$$
On the other hand the map
\begin{align}
&e:\N^n \rightarrow \wp(\C^n)\\
&\;\;\;\;\;\;\alpha \mapsto e_\alpha(z):=z_1^{\alpha_1}\cdot z_2^{\alpha_2}\cdot \dots \cdot z_n^{\alpha_n}\label{alphadef}
\end{align}
is a isomorphism having the property that $\deg e_\alpha(z)=|\alpha|.$

Thus, if we denote by $e_i(z)$ the $i$-th monomial function $e_{\alpha(i)}(z)$, we have
$$\wp^k(\C^n)=\Span\{e_i(z),1\leq i\leq N_k\}=:\Span \mathcal M_n^k,$$
where 
$$N_k:=\ddim\wp^k(\C^n)={n+k \choose n}.$$
From now on we will refer to $\mathcal M_n^k$ as the graded lexicographically ordered monomial basis of degree $k.$

For any array of points $\{\bs z_1,\dots \bs z_{N_k}\}\subset E^{N_k}$ we introduce the \emph{Vandermonde determinant} of order $k$ as
$$\vdm_k(\bs z_1,\dots \bs z_{N_k}):=\det[e_i(\bs z_j)]_{i,j=1,\dots,N_k}.$$

If a set $\{\bs z_1,\dots,\bs z_{N_k}\}$ of points of $E$ satisfies
$$\left|\vdm_k(\bs z_1,\dots \bs z_{N_k})\right|=\max_{\bs \z_1,\dots\bs \z_{N_k}\in E}\left|\vdm_k(\bs \z_1,\dots \bs \z_{N_k})\right|$$
it is said to be a \emph{array of Fekete points} for $E.$ Clearly Fekete points do not need to be unique. Zaharjuta proved in his seminal work \cite{Za74I} that the sequence
$$\delta_k(E):=\left( \max_{\bs \z_1,\dots\bs \z_{N_k}\in E}\left|\vdm_k(\bs \z_1,\dots \bs \z_{N_k})\right|\right)^{\frac{n+1}{nk{N_k}}}$$
does have limit and defined, by analogy with the case $n=1$, the \emph{transfinite diameter} of $E$ as 
\begin{equation} \label{TdDef}
\delta(E):=\lim_k \delta_k(E).
\end{equation}
It turns out that the condition $\delta(E)=0$ characterize pluripolar subsets of $\C^n$ as it characterize polar subset of $\C.$ 

Let $\{\bs z^{(k)}\}_{k\in \N}:=\{(\bs z_1^{(k)},\dots,\bs z_{N_k}^{(k)})\}_{k\in \N}$ be a sequence of Fekete points for the compact set $E$, we consider the canonically associated sequence of uniform probability measures $\nu_k:=\sum_{j=1}^{N_k}\frac 1 {N_k} \delta_{\bs z_j^{(k)}}.$

Berman and Boucksom \cite{BeBo10} showed\footnote{The results of \cite{BeBo10} and \cite{BeBoNy11} hold indeed in the much more general setting of high powers of a line bundle on a complex manifold.} that the sequence $\nu_k$ converges weak star to the pluripotential equilibrium measure $\mu_E$ as it happens in the case $n=1.$ We will use both this result (in Section \ref{SectionEquilibrium}) and a remarkable intermediate step of its proof (in Section \ref{SectionExtremalfunction} and Section \ref{SectionTransfinite}) termed \emph{Bergman Asymptotic}, see \eqref{bergmanasymptotic} below.

\subsection{Admissible meshes and Bernstein Markov measures}
We recall that a compact set $E \subset\R^n$ (or $\C^n$) is said to be polynomial determining if any polynomial vanishing on $E$ is necessarily the null polynomial.

Let us consider a polynomial determining compact set $E \subset\R^n$ (or $\C^n$) and let
$A_k$ be a subset of $E$. If there exists a positive constant $C_k$ such that for any
polynomial $p \in \wp^k(\C^n)$ the following inequality holds
\begin{equation}
\|p\|_E\leq C_k \|p\|_{A_k},\label{amdef}
\end{equation}
then $A_k$ is said to be a \emph{norming set} for $\wp^k (\C^n ).$ 

Let $\{A_k \}$ be a sequence of norming sets for $\wp^k(\C^n)$ with constants $\{C_k \}$, suppose that both $C_k$ and $\Card(A_k)$ grow at most polynomially with $k$ (i.e., $\max\{C_k , \Card(A_k )\}$ $= \mathcal O(k^s)$ for a suitable $s \in \N$), then $\{A_k\}$ is said to be a \emph{weakly admissible mesh} (WAM) for $E$; see\footnote{The original definition in \cite{CL08} is actually slightly weaker (sub-exponential growth instead of polynomial growth is allowed), here we prefer to use the present one which is now the most common in the literature.} \cite{CL08}. Observe that necessarily
\begin{equation}\label{EquationMinimalCardinality}
\Card A_k\geq N_k:=\ddim\wp^k(\C^n)={{k+n}\choose{k}}=\mathcal O(k^n)
\end{equation}
since a (W)AM $A_k$ is $\wp^k(\C^n )$-determining by definition.

If $C_k \leq C$ $\forall k$, then $\{A_k \}_\N$ is said an \emph{admissible mesh} (AM) for $E$; in the sequel, with a little abuse of notation, we term (weakly) admissible mesh not only the whole sequence but also its $k$-th element $A_k$. When $\Card(A_k) =O(k^n)$, following Kro\'o \cite{Kr11}, we refer to $\{A_k\}$ as an \emph{optimal admissible mesh}, since this grow rate for the cardinality is the minimal one in view of equation \eqref{EquationMinimalCardinality}.

Weakly admissible meshes enjoy some nice properties that can be also used together with classical polynomial inequalities to construct such sets. For instance, WAMs are stable under affine mappings, unions and cartesian products and well behave under polynomial mappings. Moreover any sequence of interpolation nodes whose Lebesgue constant is growing sub-exponentially is a WAM.

It is worth to recall other nice properties of (weakly) admissible meshes. Namely, they enjoy a stability property under smooth mapping and small perturbations both of $E$ and $A_k$ itself; \cite{PiVi13}. For a survey on WAMs we refer the reader to \cite{BoCaLeSoVi11}. 

Weakly admissible meshes are related to Fekete points. For instance assume a Fekete triangular array $\{\bs z^{(k)}\}=\{(z_1^{(k)},\dots,z_{N_k}^{(k)})\}$ for $E$ is known, then setting
$A_k:=\bs z^{(k\log k)}$ for all $k\in \N$ we obtain an admissible mesh for $E$; \cite{LevSur}.

Conversely, if we start with an admissible mesh $\{A_k\}$ for $E$ it has been proved in \cite{BoDeSoVi10} that it is possible to extract (by numerical linear algebra) a set $\bs z^{(k)}:=\{z_1^{(k)},\dots,z_{N_k}^{(k)}\}\subset A_k$ from each $A_k$ such that the sequence    $\{\bs z^{(k)}\}$ is an \emph{asymptotically Fekete} sequence of arrays, i.e., 
\begin{equation}\label{asymptoticallyfeketedefinition}
\left|\vdm_k(z_1^{(k)},\dots,z_{N_k}^{(k)})\right|^{\frac{n+1}{nkN_k}}\to \delta(E)
\end{equation}
as happens for Fekete points. By the deep result of Berman and Boucksom \cite{BeBoNy11} (and some further refining, see \cite{BlBoLeWa10}) it follows that the sequence of uniform probability measures $\{\nu_k\}$ canonically associated to $\bs z^{(k)}$ converges weak star to the pluripotential equilibrium measure.

In \cite{BeBoNy11} authors pointed out, among other deep facts, the relevance of a class of measures for which a strong comparability of uniform and $L^2$ norms of polynomials holds, they termed such measures \emph{Bernstein Markov} measures. Precisely, a Borel finite measure $\mu$ with support $S_\mu\subseteq E$ is said to be a Bernstein Markov measure for $E$ if we have
\begin{equation}\label{equationBMdef}
\limsup_k \left(\sup_{p\in \wp^k\setminus \{0\}}\frac{\|p\|_E}{\|p\|_{L^2_\mu}}\right)^{1/k}\leq 1.
\end{equation}

Let us denote by $\{q_j(z,\mu)\}$, $j=1,\dots,N_k$ the orthonormal basis (obtained by Gram-Schmidt orthonormalizaion starting by $\mathcal M_n^k$) of the space $\wp^k$ endowed by the scalar product of $L^2_\mu.$ The reproducing kernel of such a space is $K_k^\mu(z,\bar z):=\sum_{j=1}^{N_k}\bar{q_j}(z,\mu) q_j(z,\mu)$, we consider the related \emph{Bergman function}
$$B_k^\mu(z):=K_k^\mu(z,z)=\sum_{j=1}^{N_k}|q_j(z,\mu)|^2.$$
As a side product of the proof of the asymptotic of Fekete points Berman Boucksom and Nystr\"om deduces the so called \emph{Bergman Asymptotic}
\begin{equation}\label{bergmanasymptotic}
\frac{B_k^{\mu}}{N_k}\mu\rightharpoonup^*\mu_E,
\end{equation}
for any positive Borel measure $\mu$ with support on $E$ and satisfying the Bernstein Markov property.

Note that, by Parseval Identity, the property \eqref{equationBMdef} above can be rewritten as
\begin{equation}\label{BMprop}
\limsup_k\|B_k^\mu\|_E^{1/(2k)}\leq 1.
\end{equation}  

Bernstein Markov measures are very close to the \textbf{Reg} class defined (in the complex plane) by Stahl and Totik and studied in $\C^n$ by Bloom \cite{B97}. Recently Bernstein Markov measures have been studied by different authors, we refer the reader to \cite{BlLePiWi15} for a survey on their properties and applications.

Our methods in the next sections rely on the fact \emph{admissible meshes are good discrete models of Bernstein Markov measures}; let us illustrate this.
From now on 
\begin{itemize}
\item we assume $E\subset \C^n$ to be a compact $\mathcal L$-regular set and hence polynomial determining,
\item we denote by $\mu_k$ the uniform probability measure supported on $A_k=\{z_1^{(k)},\dots, z_{M_k}^{(k)}\}$, i.e.,
$$\mu_k:=\frac{1}{\Card A_k}\sum_{j=1}^{\Card A_k}\delta_{z_j^{(k)}},$$
\item we denote by $B_k(z)$ the function $B_k^{\mu_k}(z)$ and by $K_k(z,\z)$ the function $K_k^{\mu_k}(z,\z).$
\end{itemize}

Assume that an admissible mesh $\{A_k\}$ of constant $C$ for the compact polynomial determining set $E\subset\C^n$ is given. Now pick $\hat z\in E$ such that $B_k(\hat z)=\max_E B_k,$ we note that
$$B_k(\hat z)=\sum_{j=1}^{N_k}c_j q_j(\hat z,\mu):=p(\hat z),\;\; c_j:= \bar{q_j}(\hat z,\mu).$$
By Parseval inequality we have 
$$\|p\|_E\leq \left(\sum_{j=1}^{N_k} |c_j|^2\right)^{1/2}\max_{z\in E} \left(\sum_{j=1}^{N_k} |q_j( z,\mu)|^2\right)^{1/2}=B_k(\hat z)=\|B_k\|_E. $$
Therefore we an write
$$\|B_k\|_E=p(\hat z)\leq\|p\|_E\leq C\|p\|_{A_k}\leq C \sqrt{B_k(\hat z)} \|B_k\|_E^{1/2},$$
thus
\begin{equation}\label{Bergmanestimate}
\|B_k\|_E\leq C^2\|B_k\|_{A_k}.
\end{equation}
On the other hand for any polynomial $p\in \wp^k$ we have 
$$\|p\|_E\leq C\|p\|_{A_k}\leq C\sqrt{\Card A_k}\|p\|_{L^2_{\mu_k}}.$$
Recall that it follows by the definition of (weakly) admissible meshes that $(C^2\Card A_k)^{1/(2k)}\to 1.$

Thus, \emph{the sequence of probability measures associated to the mesh has the property}
\begin{equation}\label{AMBMprop}
\limsup_k\|B_k\|_E^{1/(2k)}\leq \limsup_k(C\sqrt{\Card A_k})^{1/k}= 1,
\end{equation}  
which closely resembles \eqref{BMprop}.

Conversely, assume $\{\mu_k\}$ to be a sequence of probabilities on $E$ with $\Card \support \mu_k=\mathcal O(k^s)$ for some $s$, then we have
$$\|p\|_E\leq \sqrt{\|B_k\|_E}\|p\|_{\support \mu_k},\;\forall p\in \wp^k.$$
Therefore, if $\|B_k\|_E=\mathcal O(k^t)$ for some $t$, the sequence of sets $\{\support \mu_k\}$ is a weakly admissible mesh for $E$.
\section{Approximating the extremal function}\label{SectionExtremalfunction}
\subsection{Theoretical results}
In this section we introduce certain sequences of functions, namely $u_k,v_k,\tilde u_k$ and $\tilde v_k$, that can be constructed starting by a weakly admissible mesh, all of them having the property of local uniform convergence to $V_E^*,$ provided $E$ is $\mathcal L$-regular.
\begin{theorem}\label{TheoremAMAsymptotics}
Let $E\subset \C^n$ be a compact $\mathcal L$-regular set and $\{A_k\}$ a weakly admissible mesh for $E$, then, uniformly in $\C^n$, we have 
\begin{align}
\lim_k v_k:=&\lim_k \frac 1 {2k}\log B_k=V_E^*\label{bergmanasymptotic},\\
\lim_k u_k:=&\lim_k \frac 1 k\log\int_E|K_k(\cdot,\z)|d\mu_k(\z)=V_E^*\label{kernelasymptotic}.
\end{align}
\end{theorem}
\begin{proof}
We first prove \eqref{bergmanasymptotic}, for we introduce 
\begin{eqnarray*}
&\mathcal F_E^{(k)}&:=\{p\in \wp^k:\|p\|_{E}\leq 1\}\\
&\log\Phi_E^{(k)}(z)&:=\sup\left\{\frac 1 k \log|p(z)|,p\in \mathcal F_E^{(k)}\right\}.
\end{eqnarray*}
The sequence of function $\Phi_E^{(k)}$ has been defined by Siciak and has been shown to converge to $\exp \tilde V_E^*$ (see equation \eqref{EquationSiciakDef}) for $E$ $\mathcal L$-regular, moreover we have $V_E^*\equiv \tilde V_E^*;$  \cite{Si81}, see also \cite{Si62}.

Let us denote by $\mathcal F^{(k)}_2$ the family $\{p\in \wp^k: \|p\|_{L^2_{\mu_k}}\leq 1\}$ we notice that, due to the Parseval Identity, we have 
$$B_k(z)=\sup_{p\in \mathcal F^{(k)}_2}|p(z)|^2.$$
Let us pick $p\in \mathcal F^{(k)}_2$, we have $\|p\|_{E}\leq\sqrt{\|B_k\|_E}  \|p\|_{L^2_{\mu_k}}$ for the reason above, thus $q:=p \|B_k\|_E^{-1/2}\in \mathcal F_E^{(k)}.$

Hence
$$\log\Phi_E^{(k)}(z)\geq \frac 1 k\log |q(z)|=\frac 1 k \log|p(z)|-\frac 1{2k}\log\|B_k\|_E, \forall p\in \mathcal F^{(k)}_2.$$
It follows that
$$\log\Phi_E^{(k)}(z)+\frac 1{2k}\log\|B_k\|_E\geq v_k(z).$$

On the other hand, since $\mu_k$ is a probability measure, we have $\|p\|_E\geq \|p\|_{L^2_{\mu_k}}$ for any polynomial. Hence if $p\in \mathcal F_E^{(k)}$ it follows that $p\in \mathcal F_2^{(k)}.$ Thus $v_k(z)\geq\log \Phi_E^{(k)}(z).$ Therefore we have
$$\log\Phi_E^{(k)}(z)+\frac 1{2k}\log\|B_k\|_E\geq v_k(z)\geq \log \Phi_E^{(k)}(z).$$

Note that we have $\limsup_k \|B_k\|_E^{1/2k}\leq 1$ since $\{A_k\}$ is weakly admissible (see equation \eqref{AMBMprop}), hence we can conclude that locally uniformly we have
\begin{align*}
V_E^*(z)&\leq \liminf_k \left(\log\Phi_E^{(k)}(z)-\frac 1{2k} \log\|B_k\|_E \right)\\
&\leq \liminf v_k(z)\leq \limsup v_k(z)\\
&\leq \limsup_k \log\Phi_E^{(k)}(z)=V_E^*(z).
\end{align*}
This concludes the proof of \eqref{bergmanasymptotic}, let us prove \eqref{kernelasymptotic}.

It follows by Cauchy-Schwarz and Holder Inequalities and by $\int B_k d\mu_k=N_k$ that
\begin{align*}
&\int |K_k(z,\z)|d\mu_k(\z)\\
\leq& \int\left( \sum_{j=1}^{N_k}|q_j(z,\mu_k)|^2\right)^{1/2}
 \left( \sum_{j=1}^{N_k}|q_j(\z,\mu_k)|^2\right)^{1/2} d\mu_k(\z)\\
&\leq \left\|\sqrt{B_k(\z)}\right\|_{L^2_{\mu_k}} \cdot \sqrt{B_k(z)}\leq N_k^{1/2} \sqrt{B_k(z)}.
 \end{align*} 
Thus it follows that
\begin{equation}\label{upperestimateu}
u_k(z)\leq \frac 1{2k}\log[N_k]+ v_k(z)\; \text{ uniformly in } \C^n.
\end{equation}
On the other hand, for any $p\in \wp^k$ we have
\begin{align*}
|p(z)|&=\left|\langle K_k(z,\z);p(\z)\rangle_{L^2_{\mu_k}}\right|=\left|\int K_k(z,\z) p(\z)d\mu_k(\z)\right|\\
&\leq \|p\|_{L^\infty_{\mu_k}}\int\left| K_k(z,\z)\right| d\mu_k(\z)\\
&\leq \|p\|_{E}\int |K_k(z,\z)| d\mu_k(\z),
\end{align*}
hence, using the definition of Siciak function,
$$\int |K_k^{\mu_k}(z,\z)| d\mu_k(\z)\geq \sup_{p\in \wp^k\setminus\{0\}}\frac{|p(z)|}{\|p\|_E}=(\Phi_E^{(k)})^k.$$
Finally, using \eqref{upperestimateu}, we have
$$\log \Phi_E^{(k)}(z)\leq u_k(z)\leq v_k(z)+\log N_k^{1/(2k)},$$
uniformly in $\C^n,$ this concludes the proof of \eqref{kernelasymptotic} since $N_k^{1/(2k)}\to 1$ and both $v_k$ and $\log \Phi_E^{(k)}$ converge to $V_E^*$ uniformly in $\C^n.$
\end{proof}
It is worth to notice that both $u_k$ and $v_k$ are defined in terms  of orthonormal polynomials with respect to $\mu_k$, hence they can be computed with a finite number of operations at any point $z\in \C^n$, indeed we have
\begin{equation}\begin{split}
v_k(z)=\frac{1}{k}\log\int|K_k(z,\z)|d\mu_k(\z)\\
=\frac{1}{k}\log\left[\frac 1{\Card A_k}\sum_{h=1}^{\Card A_k}\left|\sum_{j=1}^{N_k}q_j(z,\mu_k)\bar q_j(\z_h,\mu_k)\right|  \right].
\end{split}
\end{equation}
Also we note that Theorem \ref{TheoremAMAsymptotics} can be understood as a generalization of the original Siciak statement \cite[Th. 4.12]{Si81}. Indeed, if we take $A_k=\{z_1^{(k)},\dots, z_{N_k}^{(k)}\}$ a set of Fekete points of order $k$ for $E$ we get $q_j(z,\mu_k)=\sqrt{N_k}\ell_{j,k}(z)$, where $\ell_{j,k}(z)$ is the $j$-th Lagrange polynomial, hence we have
\begin{align*}
&\frac 1{N_k}\sum_{h=1}^{N_k}\left|\sum_{j=1}^{N_k}q_j(z_,\mu_k)\bar q_j(\z_h,\mu_k)\right|=\frac 1{\sqrt{N_k}}\sum_{h=1}^{N_k}\left|\sum_{j=1}^{N_k}q_j(z_,\mu_k) \delta_{|j-h|}\right|\\
=&\frac 1{\sqrt{N_k}}\sum_{h=1}^{N_k}\left|q_h(z_,\mu_k) \right|=\sum_{h=1}^{N_k}|\ell_{h,k}(z)|=:\Lambda_{A_k}(z).
\end{align*}
Here $\Lambda_{A_k}(z)$ is the Lebesgue function of the interpolation points $A_k.$ Therefore, for $A_k$ being a Fekete array of order $k$, we have $u_k(z)=\log \left(\Lambda_{A_k}(z)\right)^{1/k}$, this is precisely $\exp\left( k\Phi_k^{(2)}(z)\right)$ in the Siciak notation.

In Section \ref{SectionEquilibrium} we will deal with measures of the form $\nu_k:=\frac{B_k^\mu}{N_k}\mu$ for a Bernstein Markov measure $\mu$ for $E$, or, more generally $\nu_k:=\frac{B_k^{\mu_k}}{N_k}\mu_k$, where the sequence $\{\mu_k\}$ has the property $\limsup_k \|B_k^{\mu_k}\|_E^{1/2k}=1$; we refer to such a sequence $\{\mu_k\}$ as a \emph{Bernstein Markov sequence of measures.} Due to a modification of the Berman Boucksom and Nymstrom result, $\nu_k$ converges weak star to $\mu_E$ (see Proposition \ref{bergmanasymptoticonwam} below). Note that $\nu_k$ is still a probability measure since $B_k(z)\geq 0$ for any $z\in \C^n$ and 
$$\int_Ed\nu_k=\frac 1 {N_k}\int_E B_k d\mu_k=\frac 1 {N_k}\sum_{j=1}^{N_k}\|q_j(z)\|_{L^2_{\mu_k}}^2=1.$$ 

Here we point out another (easier, but very useful to our aims) property of the "weighted" sequence $\nu_k$: actually they are a Bernstein Markov sequence of measure, more precisely the following theorem holds.

\begin{theorem}\label{TheoremwAMAsymptotics}
Let $E\subset \C^n$ be a compact $\mathcal L$-regular set and $\{A_k\}$ a weakly admissible mesh for $E.$ Let us set $\tilde \mu_k:=\frac{B_k^{\mu_k}}{N}\mu_k,$ where $\mu_k$ is the uniform probability measure on $A_k.$ Then the following holds.
\begin{enumerate}[i)]
\item For any $k\in \N$ and any $z\in \C^n$
\begin{equation}\label{weightedBMS}
B_k^{\tilde \mu_k}(z)\leq \frac N {\min_E B_k} B_k^{\mu_k}(z)\leq N B_k^{\mu_k}(z). 
\end{equation}
Thus $\limsup_k \|B_k^{\tilde \mu_k}\|_E^{1/2k}=1.$
\item We have  
\begin{align}
\lim_k \tilde v_k:=&\lim_k \frac 1 {2k}\log B_k^{\tilde \mu_k}=V_E^*\label{wbergmanasymptotic},\\
\lim_k \tilde u_k:=&\lim_k \frac 1 k\log\int_E|K_k^{\tilde \mu_k}(\cdot,\z)|d\tilde \mu_k(\z)=V_E^*\label{wkernelasymptotic},
\end{align}
uniformly in $\C^n.$
\end{enumerate}
\end{theorem}
From now on we use the notations
$$\tilde B_k:=B_k^{\tilde \mu_k}(z)\;,\;\;\;\tilde \mu_k:=\frac{B_k}{N}\mu_k,$$
where $B_k$ is as above $B_k^{\mu_k}$ and $\mu_k$ will be clarified by the context.
\begin{proof}
We prove \eqref{weightedBMS}, then $\limsup_k \|\tilde B_k\|_E^{1/2k}=1$ follows immediately by $\limsup_k \|B_k\|_E^{1/2k}=1$ and $\lim_k N_k^{1/k}=1.$ The proof of \eqref{wbergmanasymptotic} and \eqref{wkernelasymptotic} is identical to the ones of Theorem \ref{TheoremAMAsymptotics} so we do not repeat them.

We simply notice that, for any sequence of polynomials $\{p_k\}$ with $\deg p_k\leq k$, we have
\begin{align*}
\|p_k\|_{L^2_{\mu_k}}^2=&\int_E|p_k(z)|^2 d\mu_k= \int_E\frac{N}{B_k(z)}|p_k(z)|^2 \frac{B_k(z)}{N}d\mu_k\\
\leq & \max_E \frac{N}{B_k}\int_E|p_k(z)|^2 \frac{B_k(z)}{N}d\mu_k= \frac{N}{\min_E B_k}\|p_k\|_{L^2_{\tilde\mu_k}}^2.
\end{align*}
Now, for any $z\in E$, we pick a sequence $\{p_k\}$ such that it maximizes (for any $k$) the ratio $(|q(z)|\,\|q\|_{L^2_{\tilde\mu_k}}^{-1})^{1/k}$ among $q\in \wp^k$ and we get
\begin{align*}
\tilde B_k(z)^{1/2}=&\frac{|p_k(z)|}{\|p_k\|_{L^2_{\tilde\mu_k}}}\leq \frac{N}{\min_E B_k} \frac{|p_k(z)|}{\|p_k\|_{L^2_{\mu_k}}}\\
\leq & \frac{N}{\min_E B_k} B_k(z)^{1/2}.
\end{align*}
Here the last inequality follows by the definition of $B_k$. Note in particular that $B_k(z)=1+|q_2(z)|^2+\dots$, thus $\frac{N}{\min_E B_k}\leq N=\mathcal O(k^n)$ and $\|\tilde B_k\|_E^{1/2k}\leq N^{1/2k}\|B_k\|_E^{1/2k}\sim \|B_k\|_E^{1/2k}$ as $k \to \infty.$
\end{proof}
\begin{remark}
We stress that the upper bound \eqref{weightedBMS} is in many cases quite rough, though sufficient to prove the convergence result \eqref{wbergmanasymptotic}. Indeed, since $\int B_k d\mu_k=1$ for any $k$, it follows that $\frac{N}{\min_E B_k}$ is always larger than  $1$, but \emph{we warn the reader that $\|\tilde B_k\|_E$ does not need to be larger than $\|B_k\|_E$ in general}. Hence the measure $\tilde \mu_k$ may be more suitable than $\mu_k$ for our approximation purposes. 
\end{remark}

\subsection{The SZEF and SZEF-BW algorithms}
The function $V_E^*$, at least for a regular set $E$, can be characterized as the unique continuous solution of the following problem
$$\begin{cases}
\ddcn u=0,& \text{in }\C^n\setminus E\\
u\equiv 0,& \text{ on } E\\
u\in \mathcal L(\C^n).
\end{cases}
$$
It is rather clear that writing a pseudo-spectral or a finite differences scheme for such a problem is a highly non trivial task, as one needs to deal with a unbounded computational domain $\C^n\setminus E$, with a positivity constraint on $\ddcn u$, and with a prescribed growth rate at infinity (both encoded by $u\in \mathcal L(\C^n)$).

Here we present the \textbf{SZEF} and \textbf{SZEF-BW algorithms} (which stands for for Siciak Zaharjuta Extremal Function and Siciak Zaharjuta Extremal Function by Bergman weight) to compute the values of the functions $u_k$ and $v_k$ (see Theorem \ref{TheoremAMAsymptotics}) and the functions $\tilde u_k$ and $\tilde v_k$ (see Theorem \ref{TheoremwAMAsymptotics}) respectively at a given set of points. In our methods both the growth rate and the plurisubharmonicity are encoded in the particular structure of the approximated solutions $u_k$ or $v_k$, while the unboundedness of $\C^n\setminus E$ does not carry any issue since all the sampling points used to build the solutions lie on $E.$  Indeed, once the approximated solution is computed on a set of points and the necessary matrices are stored, it is possible to compute $u_k$ or $v_k$ on another set of points by few very fast matrix operations; this will be more clear in a while.

To implement our algorithms we make the following assumptions.
\begin{itemize}
\item Let $E\subset \C^n$ be a compact regular set, for simplicity let us assume $E$ to be a real body (i.e., the closure of a bounded domain), but notice that this assumption is not restrictive neither from the theoretical nor from the computational point of view.
\item We further assume that we are able to compute a weakly admissible mesh $\{A_k\}=\{z_1,\dots,z_{M_k}\}$ for $E$ with constants $C_k,$ $k\in \N.$ Note that an algorithmic construction of an admissible mesh is available  in the literature for several classes of sets \cite{PiVi14,PiVi14b,Pi16}, since the study of (weakly) admissible meshes is attracting certain interest during last years.
\item We assume $n=2$ to hold the algorithm complexity growth.
\end{itemize}
The implementation is based on the following choices.
\begin{itemize}
\item Let us fix a \emph{computational grid} 
$$\{\z_1^{(1)},\dots,\z_1^{(L_1)}\}\times\{\z_2^{(1)},\dots,\z_2^{(L_2)}\}=:\Omega\subset\C^2$$
with finite cardinality $L:=L_1\cdot L_2$ and let us denote by $\Omega_E$ the (possibly empty) set $\Omega\cap E$, while by $\Omega_0$ the set $\Omega\setminus \Omega_E.$ We will reconstruct the values $u_k(\z),v_k(\z)$ $\forall \z\in \Omega$, however we will test the performance of our algorithm (see Subsection \ref{SubSecExtremalFunctioTest} below) only in term of error on $\Omega_0.$ This choice is motivated as follows. First, we have to mention the fact that the point-wise error of $u_k$ and $v_k$ exhibits two rather different behaviours depending whether the considered point $\z$ lies on $E$ or not: the convergence for $\z\in E$ is much slower. In contrast, for any regular set $E,$ the function $V_E^*$ identically vanishes on $E$, hence there is no point in trying to approximate it on $E.$ Note that the function $V_E^*(\z)\neq 0$ for any $\z\in\C^n\setminus E$, thus we can measure the error of $u_k$ of $v_k$ on $\Omega_0\subset \C^2\setminus E$ both in the absolute and in the relative sense.
\item Let $L:\C^2\rightarrow\C^2$ be the invertible affine map, mapping $A_k$ in the square $[-1,1]^2$ and defined by $P_i(z):=\frac{2}{b_i-a_i}\left(z_i-\frac{a_i+b_i}{2}\right),$ $a_i:=\min_{z\in A_k}z_i,$ $b_i:=\max_{z\in A_k}z_i.$ We denote by $T_j(z)$ the classical $j$-th Chebyshev polynomial and we set
$$\phi_j(z):=T_{\alpha_1(j)}(P_1(z))T_{\alpha_2(j)}(P_2(z)),\;\forall j\in \N, j>0,$$
where $\alpha:\N\to\N^2$ is the one defined in \eqref{alphadef}.
The set $\mathcal T^k_P:=\{\phi_j(z),\,1\leq j\leq (k+1)(k+2)/2\}$ is a (adapted Chebyshev) basis of $\wp^k(\C^2).$  We will use the basis $\mathcal T^k_P$  for computing and orthogonalizing the Vandermonde matrix of degree $k$ computed at $A_k$. This choice has been already fruitfully used, for instance in \cite{BoDeSoVi10,BoCaLeSoVi11}, when stable computations with Vandermonde matrices are needed and is on the background of the widely used matlab package ChebFun2 \cite{ToTre13}.

\end{itemize}
\subsubsection{SZEF Algorithm Implementation} $ $\\
The \textbf{first step} of SZEF is the \emph{computation of the Vandermonde matrices} with respect to the basis $\mathcal T^k_P$ 
\begin{align}
VT&:= (\phi_j(z_i))_{i=1,\dots,M_k:=\Card A_k,\,j=1,\dots,(k+1)(k+2)/2}\\
WT&:= (\phi_j(\z_i))_{i=1,\dots,L,\,j=1,\dots,(k+1)(k+2)/2}.
\end{align}
Here $VT$ is computed simply by the formula $T_h(x)=\cos(h\arccos(x))$ while for $WT$ we prefer to use the recursion algorithm to improve the stability of the computation since the points $P(\z_i),$ $\z_i\in \Omega,$ may in general lie outside of $[-1,1]^2$ or even be complex. 

The \textbf{second step} of the algorithm is the most delicate: we perform the orthonormalization of $VT$ and store the triangular matrix defining the change of basis. More precisely the orthonormalization procedure is performed by applying the QR algorithm twice (following the so called \emph{twice is enough} principle). First we apply the QR algorithm to $VT$ and we store the obtained $R_1$, then we apply the QR algorithm to $VT\setminus R_1,$ we obtain $Q$ and $R_2$ that we store. Here $\setminus$ is the matlab \emph{backslash} operator implementing the backward substitution; this is much more stable than the direct inversion of the matrix $R_1$, which may be ill conditioned. Note that $Q_{i,j}= p_j(z_i),$ where, since $Q$ is an orthogonal matrix, 
$$M_k\int p_j(z) \bar p_h(z) d\mu_k(z)=\sum_{i=1}^M  p_j(z_i) \bar p_h(z_i)=(Q Q^T)_{j,h}=\delta_{j,h}.$$
 Therefore $\sqrt{M_k}Q_{i,j}=q_j(z_i).$     

\textbf{Step three.} We \emph{compute the orthonormal polynomials} evaluated at the points $\{\z_1,\dots\z_L\}=\Omega.$ Again we prefer the backslash operator rather than the matrix inversion to cope with the possible ill-conditioning of $R_1$ and $R_2$: we compute 
$$W:=WT\setminus R_1\setminus R_2.$$
Note that $W_{i,j}=p_j(\z_i)$ and thus $q_j(\z_i)=W_{i,j}\sqrt{M_k}.$ We also compute the matrix $K= Q\cdot W^T$, here we have 
$$K_{i,h}=\sum_{j=1}^{N_k}p_j(z_i)p_j(\z_h)=\frac 1{M_k}\sum_{j=1}^{N_k}q_j(z_i)q_j(\z_h).$$

\textbf{Step 4.} Finally we  have 
\begin{equation}
(v_k(\z_1),\dots,v_k(\z_L))_{i=1,\dots,L}=V:=\left(\frac 1{2k}\log\sum_{j=1}^{N_k}M_k W_{i,j}^2\right)_{i=1,\dots,L}
\end{equation}
and
\begin{equation}
(u_k(\z_1),\dots,u_k(\z_L))_{i=1,\dots,L}=U:=\left(\frac 1 k\log\sum_{i=1}^{M_k} |K_{i,h}|  \right)_{h=1,\dots,L}
\end{equation}
\subsubsection{SZEF-BW Algorithm Implementation} $ $\\
First the \textbf{step 1} and \textbf{step 2} of SZEF algorithm are performed.

\textbf{Step 3.} The Bergman weight $\sigma=(B_k/N(z_1),\dots,B_k(z_M)/N)$ is computed by 
$$\sigma_i=\frac{M_k}{N_k}\sum_{j=1}^{N_k} Q_{i,j}^2.$$
Then the weighted Vandermonde matrix $V^{(w)}_{i,j}:=\sqrt{\sigma_i}Q_{i,j}$ is computed by a matrix product.

\textbf{Step 4.} We compute the orthonormal polynomials by another orthonormalization by the QR algorithm: $V^{(w)}=Q^{(w)}\cdot R^{(w)}$. We get
$$\tilde q_j(z_i)=\sqrt{M_k}\sqrt{\sigma_i^{-1}}V^{(w)}_{i,j}=:\tilde Q_{i,j}.$$
We also compute the evaluation of $\tilde q_j$s at $\Omega$ by
$$\tilde q_j(\z_i)=WT\setminus R_1\setminus R_2\setminus R^{(w)}=:\tilde W_{i,j}.$$

\textbf{Step 5.} We perform the step 4 of the SZEF algorithm, where $Q$ and $W$ are replaced by $\tilde Q$ and $\tilde W.$

\subsection{Numerical Tests of SZEF and SZEF-BW}\label{SubSecExtremalFunctioTest}
The extremal function $V_E^*$ can be computed analytically for very few instances as real convex bodies and sub-level sets of complex norms. For the unit real ball $B$ Lundin (see for instance \cite{Kli}) proved that
\begin{equation}\label{LundinFormula}
V_B(z)\equiv V_B^*(z)=\frac 1 2 \log h(\|z\|+|\langle z,\bar z\rangle-1|),
\end{equation}
where $h(\z):=\z+\sqrt{1-\z^2}$ denotes the inverse of the Joukowski function and maps conformally the set $\C\setminus [-1,1]$ onto $\C\setminus \{z\in \C:|z|\leq 1\}.$

Let $E$ be any real convex set, one can define the \emph{convex dual} set $E^*$ as
$$E^*:=\{x\in \R^n: \langle x,y\rangle \leq 1,\forall y\in E \}.$$
Baran \cite{Ba88,BaT89} proved, that if $E$ is a convex compact set symmetric with respect to $0$ and containing a neighbourhood of $0$, the following formula holds.
\begin{equation}\label{baranformula}
V_E^*(z)=\sup\{\log|h(\langle z,w\rangle)|,w\in extr E^*\}.
\end{equation} 
Here $extr F$ is the set of points of $F$ that are not interior points of any segment laying in $F.$

In order to be able to compare our approximate solution with the true extremal function, we build our test cases as particular instances of the Baran's Formula:
\begin{itemize}
\item test case 1, $E_4:=[-1,1]^2,$
\item test case 2, $E_m$ the real $m$-agon centred at $0,$
\item test case 3, $E_\infty:=B$, the real unit disk.
\end{itemize}
We measure the error with respect to the true solution $V_{E}$ in terms of 
$$e_1(f,V_E^*,\Omega):=\frac{1}{\Card \Omega_0}\sum_{i=1}^{\Card \Omega_0}|f(\z_i)-V_{E}(\z_i)|,$$
which should be intended as a quadrature approximation for the $L^1$ norm of the error.

Also we consider, as an approximation of the relative $L^1$ error, the quantity
$$e_1^{rel}(f,V_E^*,\Omega):=\frac{\sum_{i=1}^{\Card \Omega_0}|f(\z_i)-V_{E}(\z_i)|}{\sum_{i=1}^{\Card \Omega_0}V_{E}(\z_i)},$$
and the following approximation of the error in the uniform norm
$$e_\infty(f,V_E^*,\Omega):=\max_{\z_i\in \Omega_0} |f(\z_i)-V_{E}(\z_i)|.$$

In all the tests we performed the rate of convergence is experimentally sub-linear, i.e., $\|s_{k+1}\|/\|s_k\|\to 1$ as $k\to \infty$, where $s_k:=f_{k+1}-f_k$ for $f_k$ being one of $u_k, \tilde u_k, v_k, \tilde v_k$ and $\|\cdot\|$ being one of the pseudo-norms we used above for defining $e_1$ and $e_\infty.$ This slow convergence may be overcame by extrapolation at infinity with the \emph{vector rho algorithm} (see \cite{Bre77}). We present below experiments regarding both the original and the accelerated algorithm.

\subsubsection{Test case 1}
In this case equation \eqref{baranformula} reads as
$$V_{[-1,1]^2}^*(z):=\max_{\z\in\{-1,1\}^2}\log|h(\z)|.$$
To build an admissible mesh for $[-1,1]^2$ we can use the Cartesian product of an admissible mesh $X_k$ for one dimensional polynomials up to degree $k;$ see \cite{BoCaLeSoVi11}. We can pick $X_k:=-\cos(\theta_j),$ $\theta_j:=j \pi/\lceil m k\rceil,$ with $j=0,1,\dots, \lceil m k\rceil$.

First we compare the performance of our four approximations in terms of $L^1$ error behaviour as $k$ grows large. Also, in order to better understand the rate of convergence, we compute the ratios
  $$s_k:=\frac{e_1(f_{k+2},f_{k+1},\Omega)}{e_1(f_{k+1},f_{k},\Omega)},$$
 for $f_k$ in $\{u_k,v_k\tilde u_k,\tilde v_k\}$ and $\Omega=[-2,2]^2.$
We report the results in Figure \ref{severalmethodssquare1} and \ref{severalmethodssquare2} respectively. The profile of convergence is slow but monotone, indeed the asymptotic constants $s_k$ become rather close to $1.$ This suggest a sub-linear convergence rate.
\begin{figure}[h]
\caption{Comparison of the $e_1(\cdot,\Omega)$ error of $u_k,$ $\tilde u_k,$ $v_k$ and $\tilde v_k$ for $E=[-1,1]^2$ and $\Omega$ the $10000$ points equispaced coordinate grid in $[-2,2]^2$}
\label{severalmethodssquare1}
\begin{center}
\includegraphics[scale=0.6]{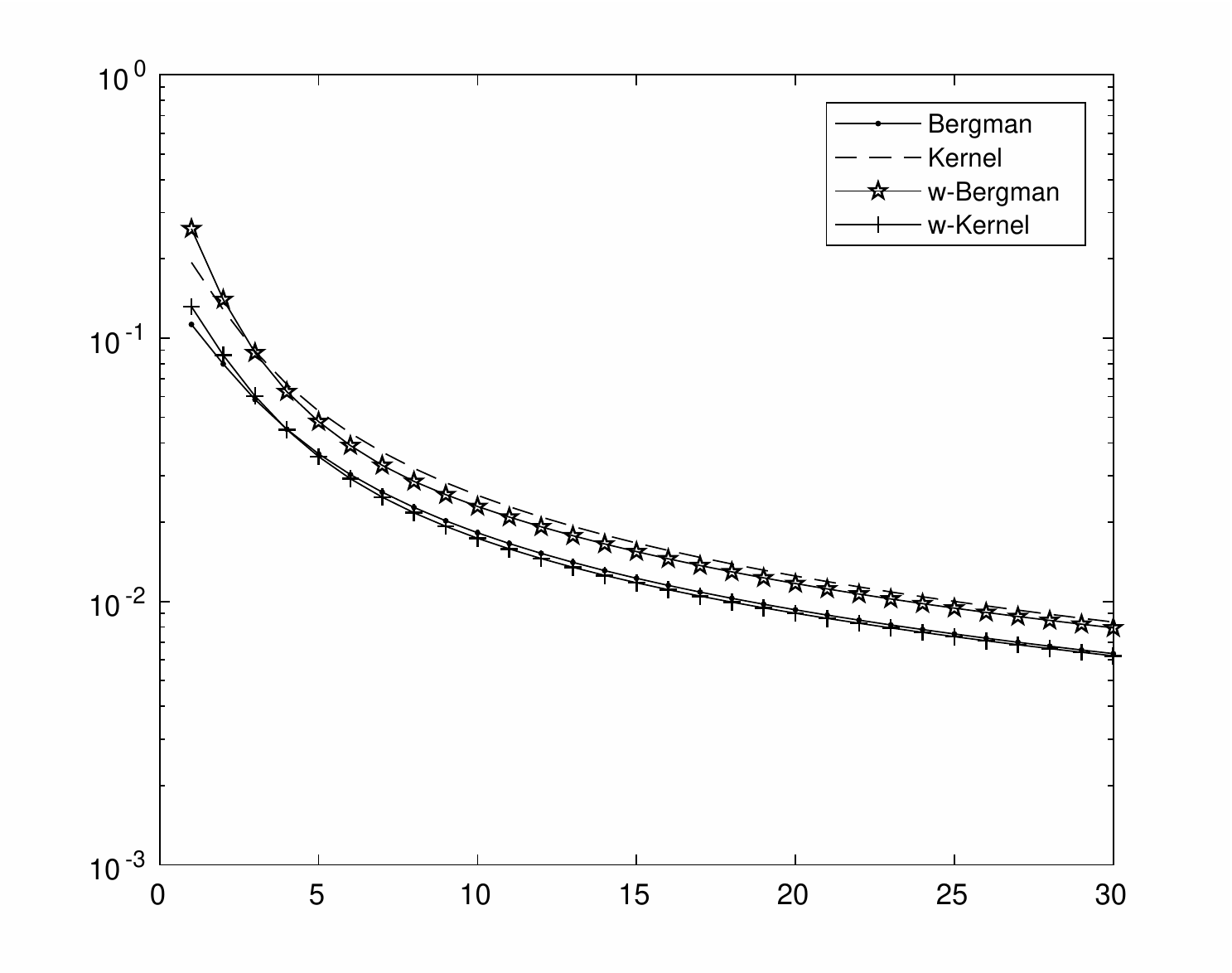}
\end{center}
\end{figure}
\begin{figure}[h]
\caption{Comparison of the $s_k$ behaviour for $f_k$ being $u_k$ (Bergman), $\tilde u_k$ (w-Bergman), $v_k$ (Kernel) or $\tilde v_k$ (w-Kernel) for $E=[-1,1]^2$ and $\Omega$ the $10000$ points equispaced coordinate grid in $[-2,2]^2.$ }
\label{severalmethodssquare2}
\begin{center}
\includegraphics[scale=0.6]{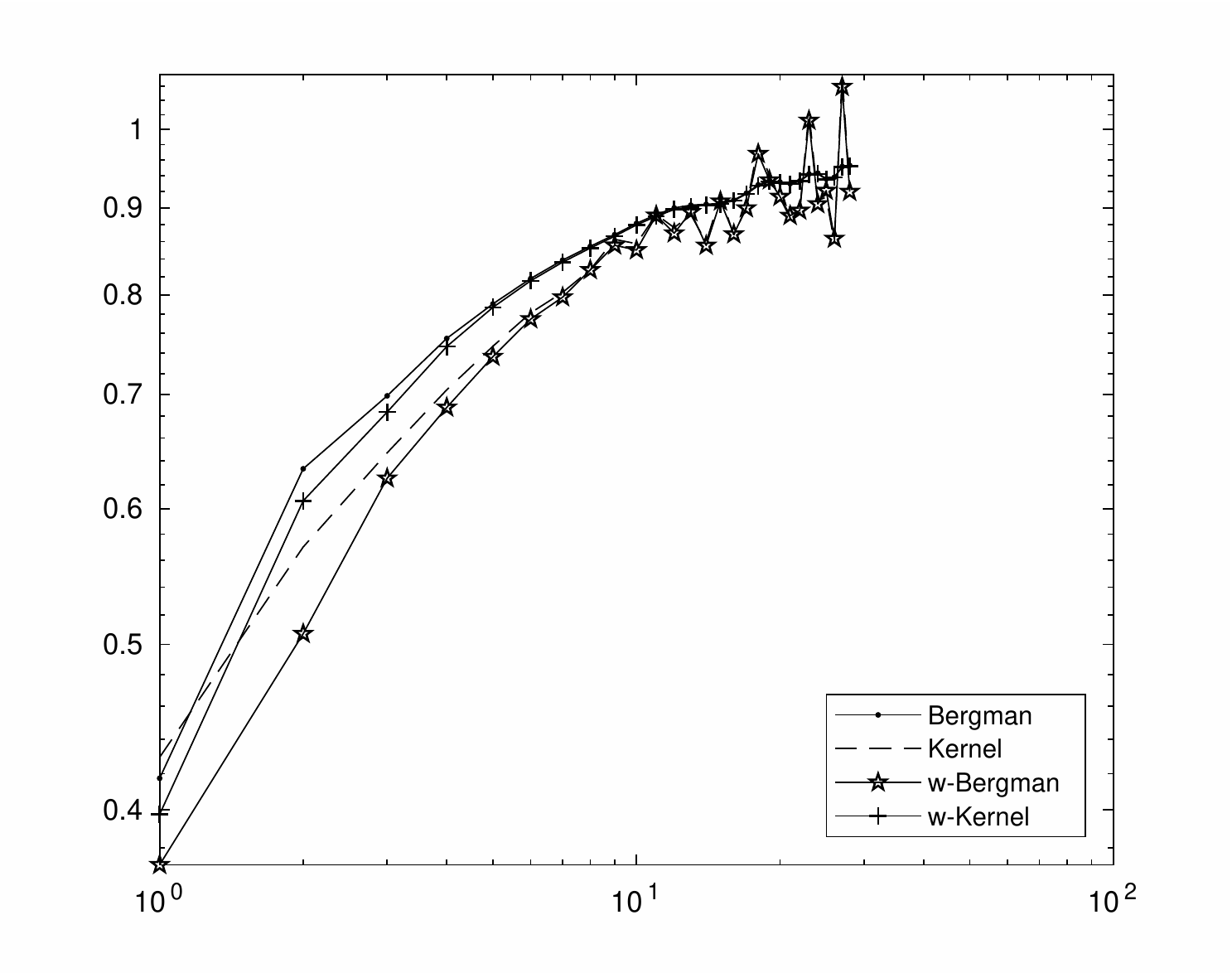}
\end{center}
\end{figure}
It is worth to say that in all the tests we made \emph{the point-wise error is much smaller far from $E$ than for points that lie near to $E$}: in the experiment reported in Figure \ref{severalmethodssquare1} and Figure \ref{severalmethodssquare2} $\Omega$ is an equispaced real grid in $[0,2]^2$, but the results are the same for larger bounds and becomes much better if $\Omega\cap E=\emptyset$ or when $\Omega$ is a purely imaginary set, i.e., $\Omega\cap \R^2=\emptyset.$

Fortunately, even if the convergence rate shown by the experiment of Figure \ref{severalmethodssquare1}, the \emph{vector rho algorithm}, see for instance \cite{Bre77}, works effectively on our approximation sequences and actually allows to produce much better approximations than the original sequences $\{u_k\},\{\tilde u_k\},\{v_k\},\{\tilde v_k\}.$ We report in Figure \ref{severaldegreessquarenear}  a comparison among the errors, on $[-2,2]^2$ and $[0,20]^2$ respectively, of the original sequence $\{u_k\}$ and the accelerated sequence produced by the vector rho algorithm, together with a linear (with respect to log-log scale) fitting of the original errors. In contrast with the sub-linear convergence enlightened above, the accelerated sequence exhibits (in the considered interval for $k$) a super-quadratic behaviour. 
\begin{figure}[h]
\caption{Log-log plot of $e_1(\tilde u_k,V_E^*,\Omega)$, $e_1^{rel}(\tilde u_k,\Omega)$ and $e_\infty(\tilde u_k,V_E^*,\Omega)$ and the same quantities for the accelerated sequence of approximations for $E=[-1,1]^2$ and $\Omega$ the $10000$ points equispaced coordinate grid in $[-2,2]^2$ (above) and in $[0,20]^2$ (below).}
\label{severaldegreessquarenear}
\begin{center}
\includegraphics[scale=0.6]{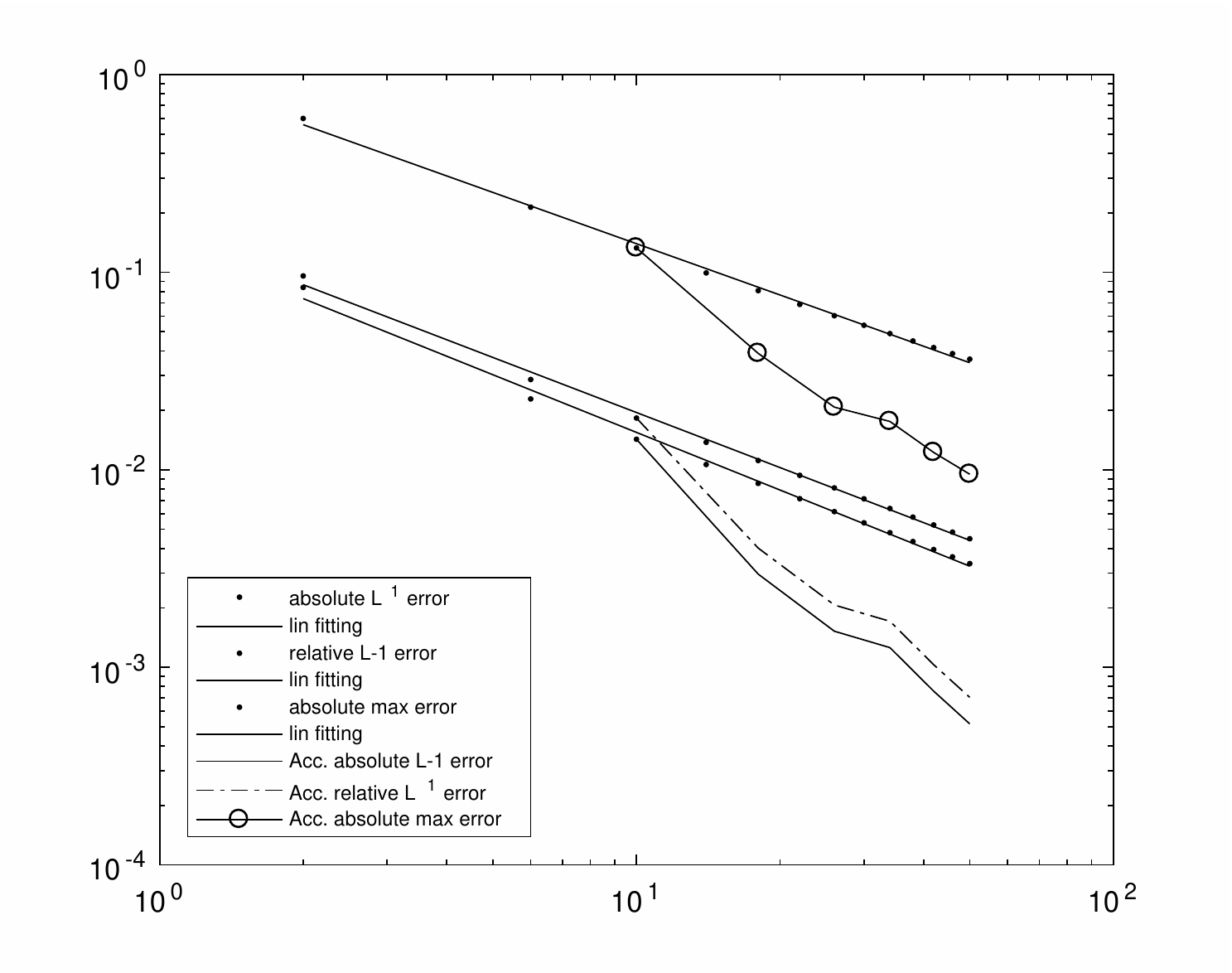} 
\includegraphics[scale=0.6]{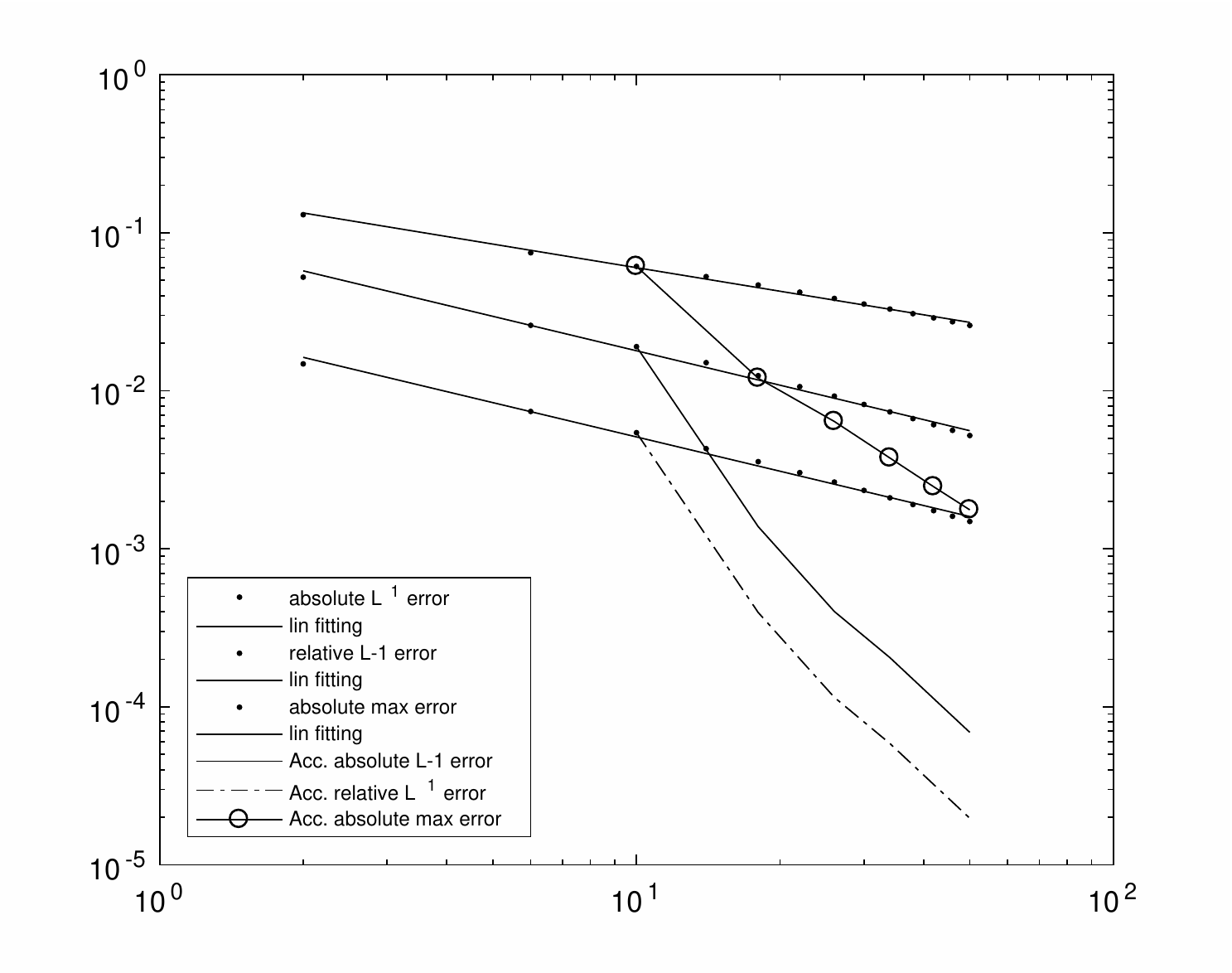}
\end{center}
\end{figure}

\subsubsection{Test case 2}
To construct a weakly admissible mesh on the real regular $m$-agon
$$E_m:=\text{conv}\left\{\left(\begin{array}{cc}1\\0\end{array}\right),\left(\begin{array}{cc}\cos\frac{2\pi}{m}\\\sin\frac{2\pi}{m}\end{array}\right),\dots\dots, \left(\begin{array}{cc}\cos\frac{2(m-1)\pi}{m}\\\sin\frac{2(m-1)\pi}{m}\end{array}\right)\right\},$$
we use the algorithm proposed in \cite{BoVi12}. First the convex polygon is subdivided in non overlapping triangles, then an admissible mesh $\{\hat A_{k,l}\}_k$ for each triangle $l,$ $l=1,2,\dots, m$ is created by the Duffy transformation; finally, the union $\{\tilde A_k\}:=\{\cup_{l=1}^m\hat A_{k,l}\}$ of all meshes (with no point repeated) is an admissible mesh for the polygon by definition.

The resulting mesh $\tilde A_k$ has good approximation properties: for instance it can be constructed in such a way to have a very small constant, however $\tilde A_k$ is not tailored to our problem. Points of $\tilde A_k$ cluster, as one could expect, at any corner of the regular polygon, but also they cluster near the edges of each triangle in the subdivision. This "spurious" clustering is not coming from the geometry of the problem, instead it is an effect of the method we used to solve it. More importantly, this issue tends to deteriorate the convergence of SZEF algorithm. Let us briefly give an insight on why this phenomena occurs in the following remark. 
\begin{remark}\label{remarkBkoscillations}
In Section \ref{sectionequilibrium} we will prove (see Proposition \ref{bergmanasymptoticonwam}) that the sequence of measures 
$$\{\hat \mu_k\}:=\left\{\frac{B_k^{\mu_k}}{N_k}\mu_k\right\}\rightharpoonup^*\mu_E,$$ 
where $\rightharpoonup^*$ denotes the convergence in the weak$^*$ topology on Borel measures. Assume for simplicity that $\mu_E$ is absolutely continuous with respect to the Lebesgue measure $m$ restricted to $E$ and consider an admissible mesh that tends to cluster points on a ball $B(z_0,r)\subset E$ for which $\frac{\mu_E}{dm}(z), z\in B(z_0,r),$ is very small.  The asymptotic above implies in particular that $B_k^{\mu_k}(z)$ will be very small for $z\in \support \mu_k\cap B(z_0,r).$

On the other hand we have $N_k^{-1}\int B_k^\mu d\mu=1$ for any measure and thus there will be some points $\hat z\in \support \mu_k\setminus B(z_0,r)$ for which $B_k^{\mu_k}(\hat z)$ is very large. Recall that the \emph{uniform} convergence of Theorem \ref{TheoremAMAsymptotics} relies on the fact that 
$$\|B_k^{\mu_k}\|_E^{1/2k}\rightarrow 1,$$
thus we should aim to get an admissible mesh $\{A_k\}$ having small constant, whose cardinality is slowly increasing, \emph{and whose Bergman function is "not too large", e.g., is "not too oscillating".} From the observation above a good heuristic to achieve such an aim is to \emph{mimic the density of the equilibrium measure.}  
\end{remark}
In order to overcome this issue we can apply two strategies.

The first one is to get rid of these extra points. To this aim we can use the AFP algorithm \cite{BoDeSoVi10} to extract a set of discrete Fekete points $A_k$ of order $2 k$ from $\tilde A_{2k}.$

Let us motivate this choice. First $A_k$ has been shown (numerically) to  be a weakly admissible mesh in many test cases, see \cite{BoDeSoVi10} and reference therein. Also, even more importantly, the resulting Bergman function is (experimentally) much less oscillating than the one that can be computed starting by $\tilde A_k$ and this turns in a smaller uniform norm on $E_m$ of such a function. Lastly, heuristically speaking in view of Remark \ref{remarkBkoscillations}, we would like to use a mesh which is mimicking the distribution of the equilibrium measure. We invite the reader to compare this last discussion with \cite{BlBoLeWa10} where the definition of optimal measures is introduced.

We can also perform a different choice which rests upon Theorem \ref{TheoremwAMAsymptotics}. Indeed the SZEF-BW algorithm uses a Bergman weighting of $\mu_k$ to prevent the phenomena we discussed in Remark \ref{remarkBkoscillations}. We tested our algorithm SZEF-BW for different values of $m$ and several choices of $\Omega,$ here we consider the case $m=6.$

Again, the convergence is quite slow but (numerically) monotone and, apart from the region of $\Omega_0$ close to $\partial_{\R^n} E,$ it is not affected by the particular choice of $\Omega;$ that is, we can appreciate numerically the \emph{global} uniform convergence proved in Theorem \ref{TheoremwAMAsymptotics}. Moreover, extrapolation at infinity is successful even in this test case. We report profiles of convergence relative to two possible choices of $\Omega$ in Figure \ref{polygonprofilenear}.   

\begin{figure}[h]
\caption{Log-log plot of $e_1(\tilde u_k,V_{E_6}^*,\Omega)$, $e_1^{rel}(\tilde u_k,\Omega)$ and $e_\infty(\tilde u_k,V_{E_6}^*,\Omega)$ and the same quantities for the accelerated sequence of approximations $\Omega$ the $10000$ points equispaced coordinate grid in $[-2,2]^2$ (above), and in $[0,20]^2$ (below).}
\label{polygonprofilenear}
\begin{center}
\includegraphics[scale=0.6]{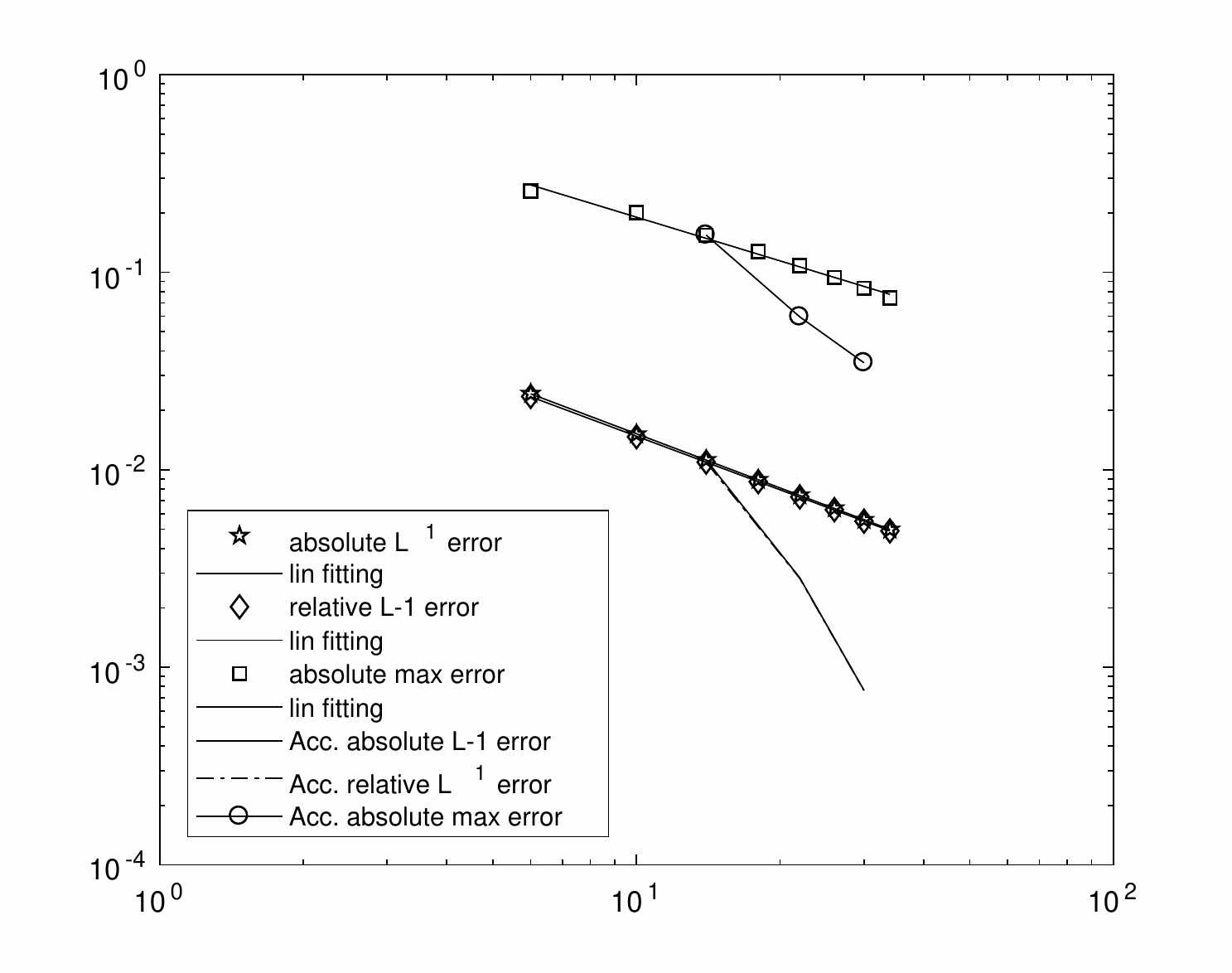}
\includegraphics[scale=0.6]{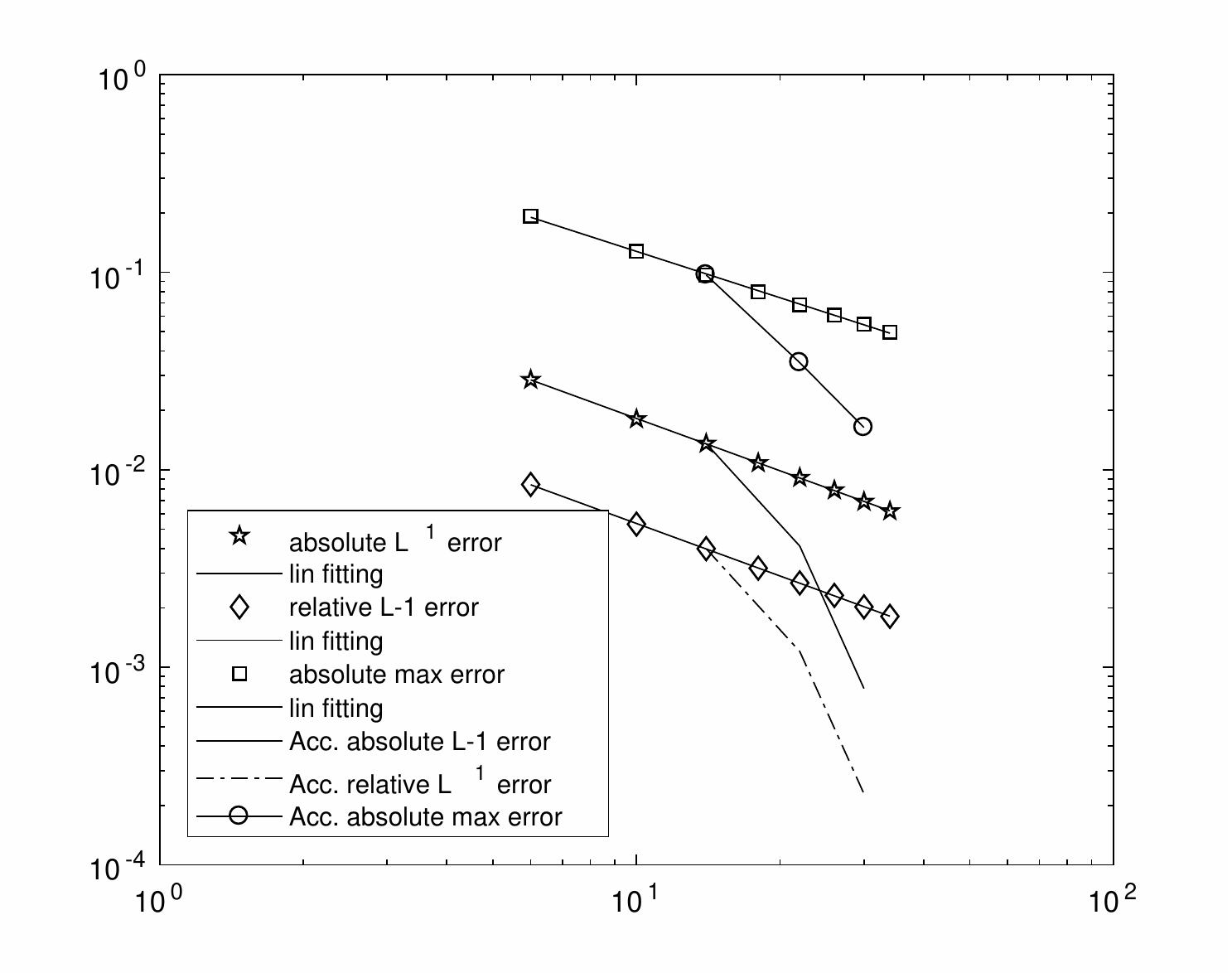}
\end{center}
\end{figure}

\subsubsection{Test case 3}
Lastly we consider $E=E_\infty:=\{(z_1,z_2): \Im(z_i)=0, \Re(z_1)^2+\Re(z_2)^2\leq 1\};$ in such a case the true solution is computed by the Lundin Formula \eqref{LundinFormula}. We can easily construct an admissible mesh of degree $k$ for the real unit disk following \cite{BoCaLeSoVi11}, for, it is sufficient to take a set of Chebyshev-Lobatto points $\{\eta_0,\eta_1,\dots,\eta_s\},$ $s>k$ and set
$$A_k:=\left\{\;\eta_h\left(\cos\left(\frac{\pi j}{2s}\right)\,,\,\sin \left(\frac{\pi j}{2s}\right)\right),\;j=0,1,\dots,s-1,\,h=0,1,\dots, s\;\right\}.$$
We report the behaviour of the error and the convergence profile in Figure \ref{diskconvergence}, again the sub-linear rate of convergence is rather evident.
\begin{figure}[h]
\caption{Comparison of $e_1(\tilde u_k,\Omega)$, $e_1^{rel}(\tilde u_k,\Omega)$ and $e_\infty(\tilde u_k,\Omega)$, their linear (in the log-log scale) fitting and the same quantities computed on the accelerated sequences of approximations for $k=6,8\dots,38$, $E_\infty$, the real unit disk, and $\Omega$ the $10000$ points equispaced coordinate grid in $[-2,2]^2$ (above) and in $[0,20]^2$ (below).}
\label{diskconvergence}
\begin{center}
\begin{tabular}{c}
\includegraphics[scale=0.6]{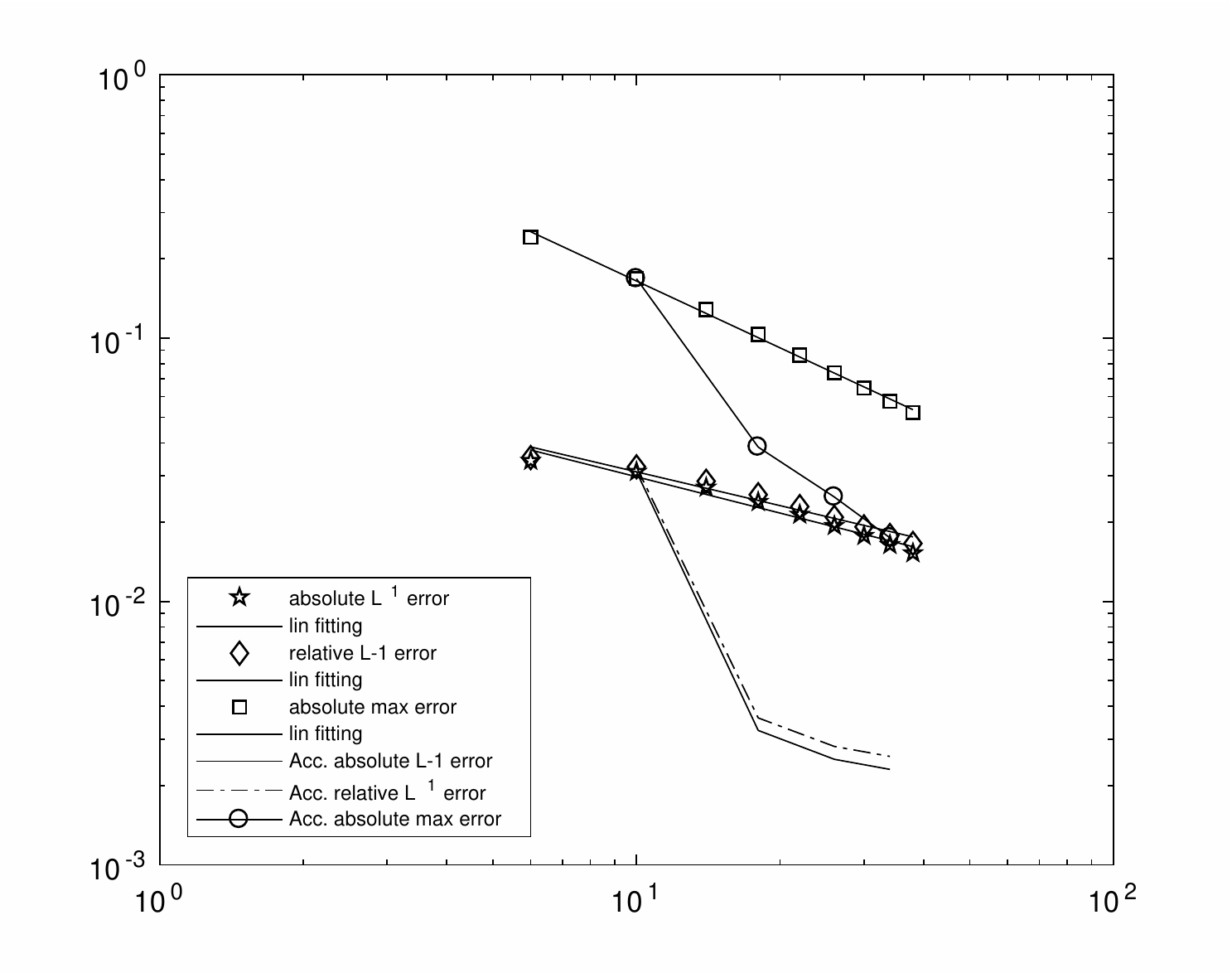}\\
\includegraphics[scale=0.6]{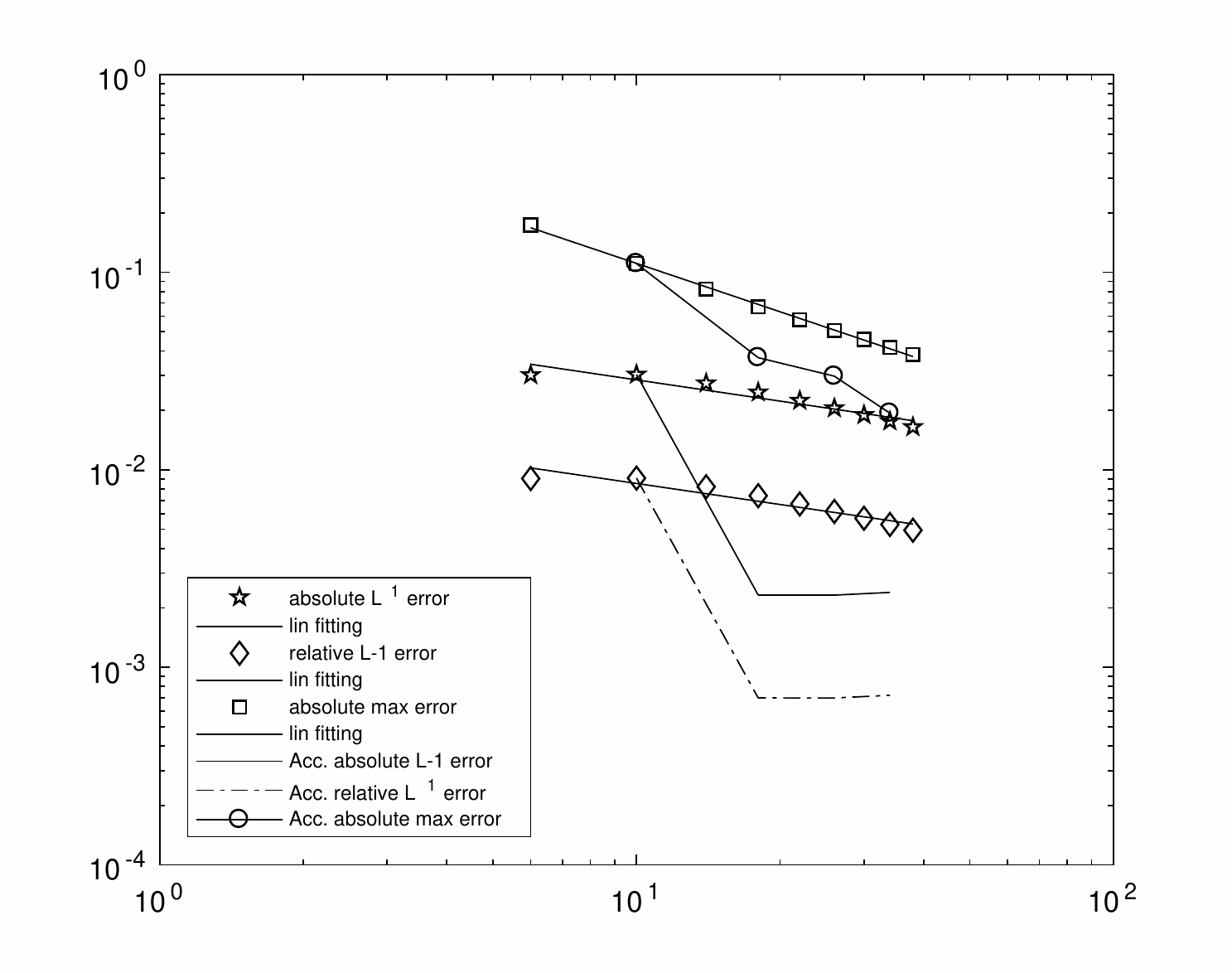}
\end{tabular}
\end{center}
\begin{tabular}{c}
\end{tabular}
\end{figure}

Also, we compare the profile of the error function $e_k(z):=|\tilde v_k(z)- V_{E_\infty}^*(z)|$ on different two real dimensional squares in Figure \ref{diskerrorfunction}. The \emph{global} uniform convergence of $\tilde v_k$ to $V_{E_\infty}^*$ theoretically proven in Theorem \ref{TheoremwAMAsymptotics} reflects on our experiments: $e_k$ is small (approximately $10^{-2}$ for $k\geq 30$) and very flat away from $E_\infty$ while it attains its maximum on $E_\infty$ with a fast oscillation near $\partial_{\R^2}E_{\infty}.$ Again, the extrapolation at infinity improves the quality of our approximation.  
\begin{figure}[h]
\caption{Profile of the error $|\tilde v_k-V_{B_\infty}^*|$ in logarithmic scale  for $k=40$, and $\Omega$ the $40000$ points equispaced coordinate grid in $[0,2]^2$ (high-left), $[100,102]^2$ (above-right), $(0.1i+[0,2])^2$ (below left) and $(i+[0,2])^2$ (below right)}
\label{diskerrorfunction}
\begin{tabular}{cc}
\includegraphics[scale=0.35]{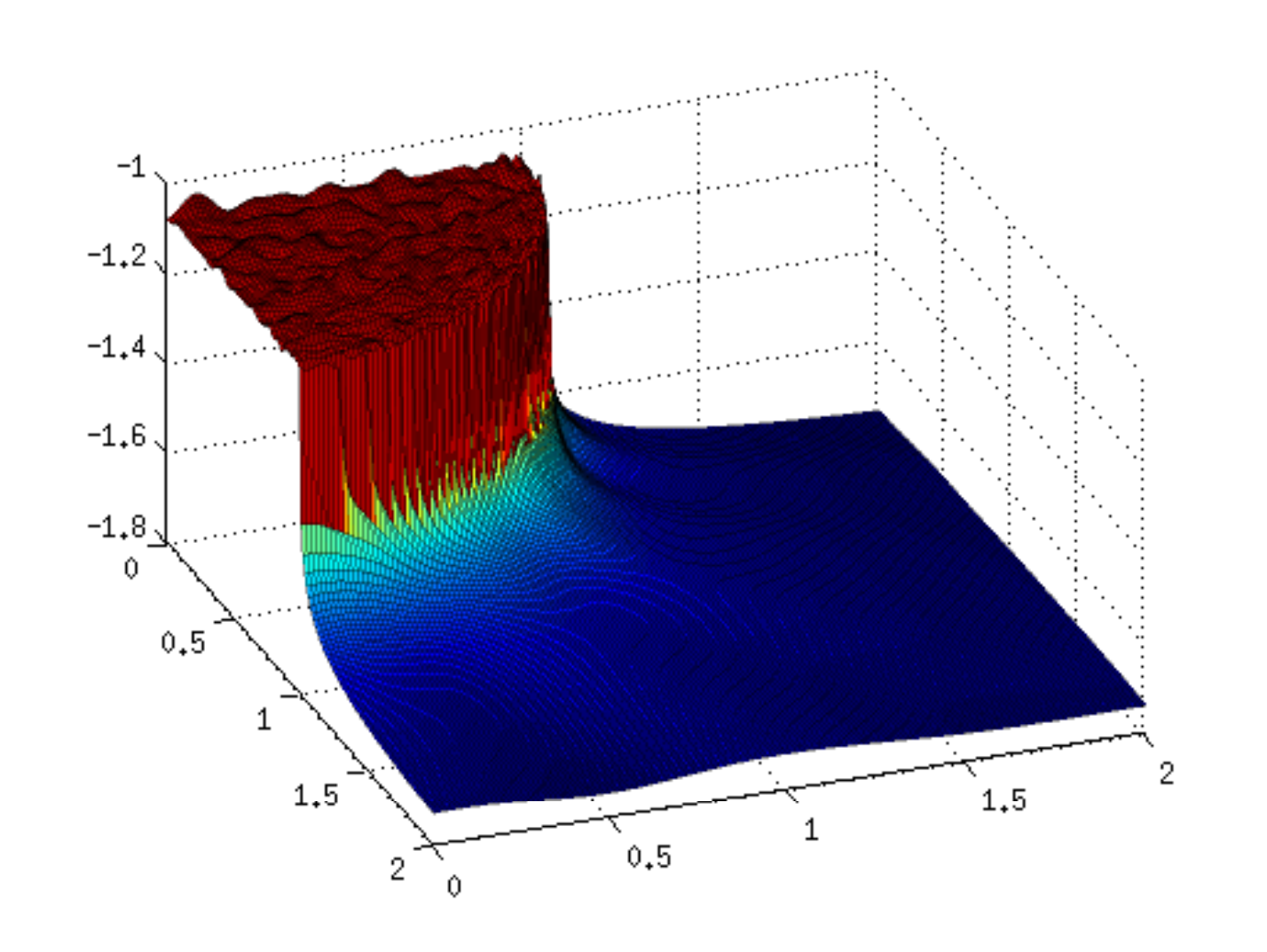}&\includegraphics[scale=0.35]{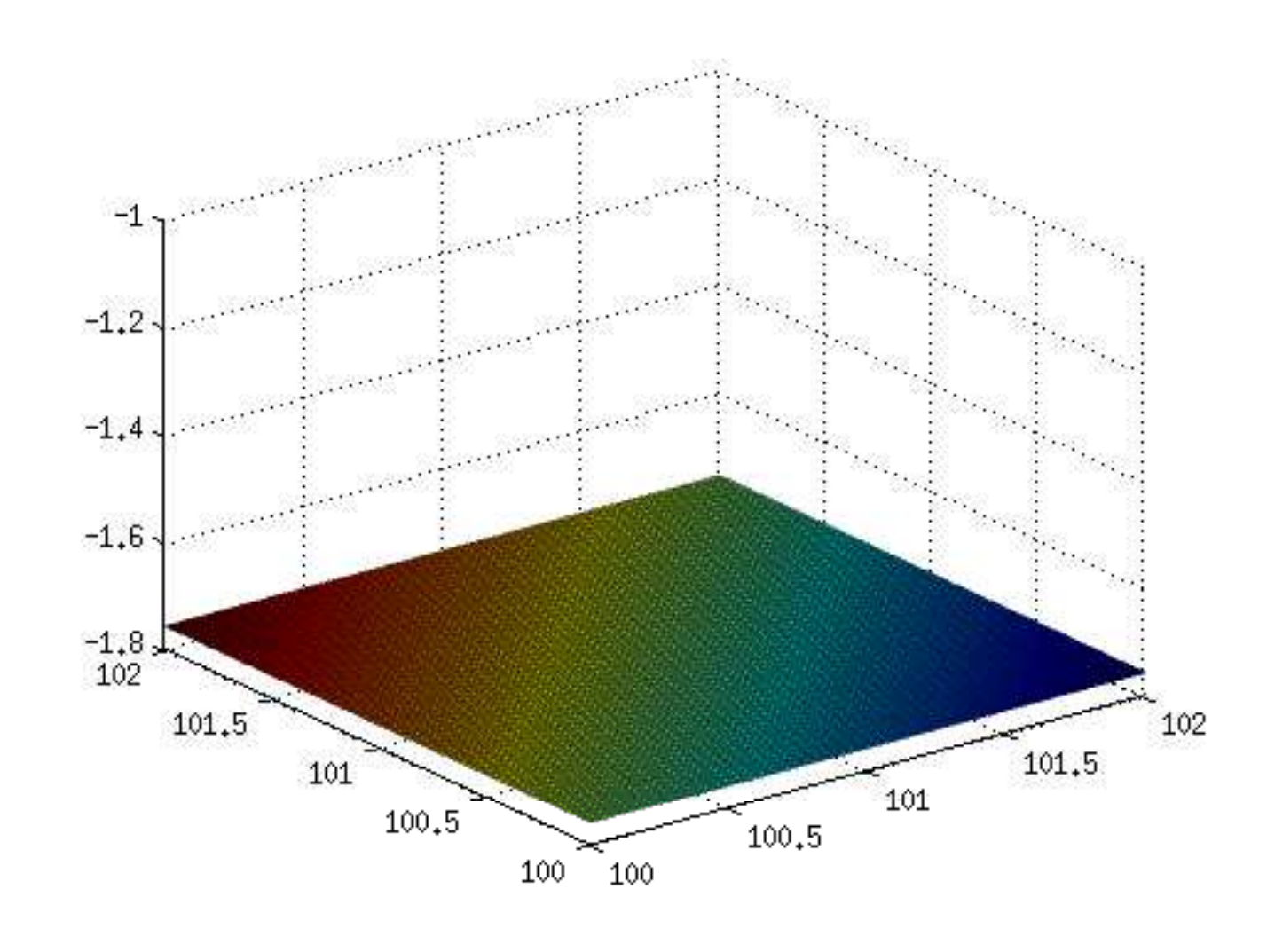}\\
\includegraphics[scale=0.35]{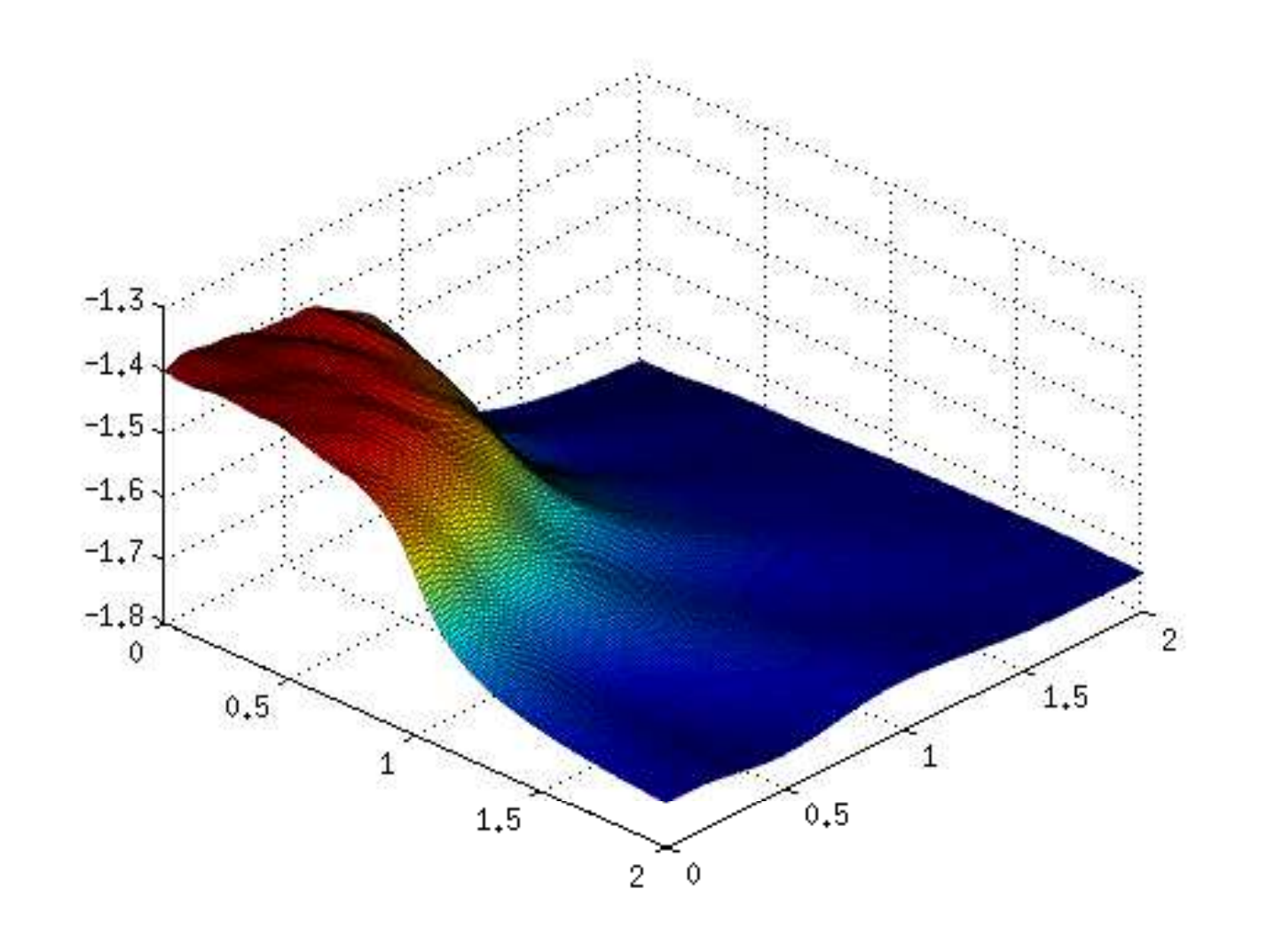}&\includegraphics[scale=0.35]{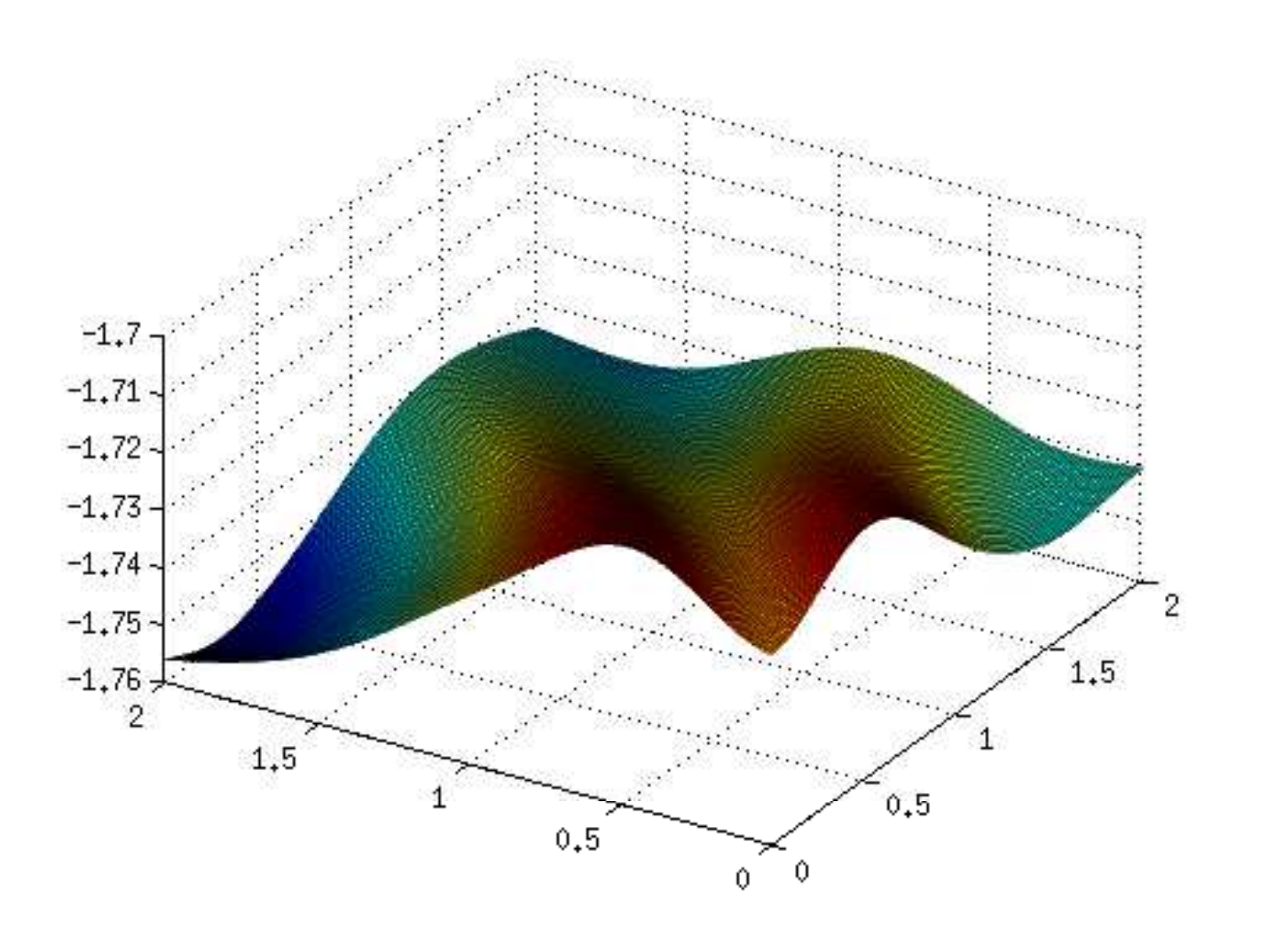}
\end{tabular}
\end{figure}

\section{Approximating the transfinite diameter}\label{SectionTransfinite}
In this section we present a method for approximating the transfinite diameter of a real compact set.
\subsection{Theoretical result}
Given any basis $\mathcal B_k=\{b_1,\dots,b_{N_k}\}$ of the space $\wp^k$ and a measure $\mu$ inducing a norm on $\wp^k$, we denote by $G_k(\mu,\mathcal B_k)$ its \emph{Gram matrix}, namely we set
$$ G_k(\mu,\mathcal B_k):=[\langle b_i;b_j\rangle_{L^2_\mu}]_{i,j=1,\dots,N_k}.$$
The hermitian matrix $G_k(\mu,\mathcal B_k)$ has square root, indeed introducing the \emph{generalized Vandermonde matrix}
$$V_k(\mu,\mathcal B_k):=[\langle q_i(\cdot,\mu);b_j\rangle_{L^2_\mu}]_{i,j=1,\dots,N_k},$$
we have $V_k(\mu,\mathcal B_k)^\textsc{H} V_k(\mu,\mathcal B_k)=G_k(\mu,\mathcal B_k).$ Note that, for $\mu$ being the uniform probability measure supported on an array of unisolvent points of degree $k$, $V_k(\mu,\mathcal B_k)$ is the standard Vandermonde matrix for the basis $\mathcal B_k$ and divided by $\sqrt{N_k}.$

We recall, see \cite{BlLe10}, the following relation between Gram determinants and $L^2$ norms of Vandermonde determinants.
\small
\begin{equation}\label{VandermondeIntegral}
\det G_k(\mu,\mathcal M_k)=\frac 1{N_k!}\int\dots \int |\det\vdm_k(\z_1,\dots,\z_{N_k})|^2d\mu(\z_1)\dots d\mu(\z_{N_k}),
\end{equation}
\normalsize
Where $\mathcal M_k$ denotes the (graded lexicographically ordered) monomial basis. 

Here is the main result this section.
\begin{theorem}\label{theoremtransfinite}
Let $E\subset \C^n$ be a compact $\mathcal L$-regular set and $\{A_k\}$ a (weakly) admissible mesh for $E$ then, denoting by $\mu_k$ the uniform probability measure on $A_k$, we have
\begin{equation}\label{equationGramasymptotic}
\lim_k \left( \det G_k(\mu_k,\mathcal M_k)\right)^{\frac{n+1}{2nkN_k}}=\delta(E). 
\end{equation}
\end{theorem}  
\begin{proof}
We use equation \eqref{VandermondeIntegral}. Since $\mu_k$ is a probability measure, it follows that 
$$\det G_k(\mu_k,\mathcal M_k)^{1/2}\leq \max_{\z_1,\dots,\z_{N_k}\in E}|\det \vdm_k(\z_1,\dots,\z_{N_k})|$$
hence
\begin{equation}
\begin{split}
&\limsup_k\left(\det G_k(\mu,\mathcal M_k)\right)^{\frac{n+1}{2nkN_k}}\\
\leq& \limsup_k\left(\max_{\z_1,\dots,\z_{N_k}\in E}|\det \vdm_k(\z_1,\dots,\z_{N_k})|\right)^{\frac{n+1}{nkN_k}}\\
\leq&\lim_k\left(\max_{\z_1,\dots,\z_{N_k}\in E}|\det \vdm_k(\z_1,\dots,\z_{N_k})|\right)^{\frac{n+1}{nkN_k}}=\delta(E).
\end{split}
\end{equation}
On the other hand, by the sampling property of admissible meshes it follows that, for any polynomial $p$ of degree at most $k$ we have $\|p\|_E\leq C\|p\|_{A_k}\leq C \sqrt{\Card A_k}\|p\|_{L^2_{\mu_k}}.$ Thus, since $\det\vdm_k (\z_1,\dots,\z_{N_k})$ is a polynomial in each variable $\z_i\in \C^n$ of degree not larger than $k$, we get
\begin{align*}
&\left( \det G_k(\mu_k,\mathcal M_k)\right)^{\frac{n+1}{2nkN_k}}=\|\dots\|\vdm_k(\z_1,\dots,\z_{N_k})\|_{L^2_{\mu_k}}\dots\|_{L^2_{\mu_k}}^{\frac{n+1}{nkN_k}}\\
\geq & \left(\frac{1}{C\sqrt{\Card A_k}}\|\dots\|\max_{z_1\in E}\vdm_k(z_1,\z_2\dots,\z_{N_k})\|_{L^2_{\mu_k}}\dots\|_{L^2_{\mu_k}}\right)^{\frac{n+1}{nkN_k}}\geq \dots\\
\geq& \left(\frac{1}{C\sqrt{\Card A_k}}\right)^{\frac{n+1}{nk}}\left( \max_{z_1,\dots,z_{N_k}\in E}|\det \vdm_k(z_1,\dots,z_{N_k})|   \right)^{\frac{n+1}{nkN_k}}
\end{align*}
Since $(\Card A_k)^{1/k}\to 1$, being $\{A_k\}$ weakly admissible, it follows that
$$\liminf_k\left( \det G_k(\mu_k,\mathcal M_k)\right)^{\frac{n+1}{2nkN_k}}\geq \delta(E).$$
\end{proof}
In principle Theorem \ref{theoremtransfinite} provides an approximation procedure for $\delta(E)$, given $\{A_k\}$, however the straightforward computation of the left hand side of \eqref{equationGramasymptotic} leads to stability issues. In the next subsection we present our implementation of an algorithm based on Theorem \ref{theoremtransfinite} and we discuss a possible way to overcome such difficulties.
\subsection{Implementation of the TD-GD algorithm}
We recall that we denote by $T_j(z)$ the classical $j$-th Chebyshev polynomial and we set $\mathcal T_k:=\{\phi_1,\dots, \phi_{N_k}\},$ where
$$\phi_j(z):=T_{\alpha_1(j)}(\pi_1 z)T_{\alpha_2(j)}(\pi_2 z),\;\forall j\in \N, j>0,$$
where $\alpha:\N\to\N^2$ is the one defined in \eqref{alphadef} and $\pi_h$ is the $h$-coordinate projection.

We denote by $V_k=V_k(A_k,\mathcal T_k)$ the Vandermonde matrix of degree $k$ with respect the mesh $A_k:=\{(x_1,y_1),\dots,(x_{M_k},y_{M_k})\}$ and the basis $\mathcal T_k$, that is
$$V_k:=\left[ \phi_j(z_i) \right]_{i=1,\dots,M_k,j=1,\dots, N_k},$$
similarly we define $W_k:=V_k(A_k,\mathcal M_k)$ where the chosen reference basis is the lexicographically ordered monomial one.

Now we notice that, setting $M_k:=\Card A_k,$
$$\langle m_\alpha,m_\beta\rangle_{L^2_{\mu_k}}=M_k^{-1}\sum_{h=1}^{M_k} (W_k)_{\alpha,h}(W_k)_{h,\beta},$$
thus we have
$$\det G_k(\mu_k)=\det \frac{W_k^\textsc{t}W_k}{M_k}.$$
The direct application of this procedure leads to a unstable computation that actually does not converge.

On the other hand, the computation of the Gram determinant in the Chebyshev basis,
$$\det  G_k(\mu_k, \mathcal T_k):=\det \frac{V_k^\textsc{t}V_k}{M_k},$$
is more stable and we have
\begin{align*}
&(\det G_k(\mu_k))^{\frac{n+1}{2nkN_k}}=\left(\det \frac{W_k^\textsc{t}W_k}{M_k}\right)^{\frac{n+1}{2nkN_k}}\\
=& \left(\det \frac{P_k^\textsc{t}V_k^\textsc{t}V_kP_k}{M_k}\right)^{\frac{n+1}{2nkN_k}}=\left(\det(P_k)\right)^{\frac{n+1}{nkN_k}} \det G_k(\mu, \mathcal T_k)^{\frac{n+1}{2nkN_k}}.
\end{align*}
Here the matrix $P_k$ is the matrix of the change of basis. Again the numerical computation of $\det P_k$ becomes severely ill-conditioned as $k$ grows large.

Instead, our approach is based on noticing that $P_k$ does not depend on the particular choice of $E$, thus we can compute the term  $\left(\det(P_k)\right)^{\frac{n+1}{nkN_k}}$ once we know $(\det G_k(\hat \mu_k))^{\frac{n+1}{2nkN_k}}$ and $(\det \tilde G_k(\hat \mu_k))^{\frac{n+1}{2nkN_k}}$ for a particular $\hat\mu_k$  which is a Bernstein Markov measure for $\hat E\subseteq [-1,1]^2$ as
\begin{equation}\label{factorchangeofbasis}
\left(\det(P_k)\right)^{\frac{n+1}{nkN_k}}=\left(\frac{\det G_k(\hat \mu_k)}{\det\tilde  G_k(\hat \mu_k)}\right)^{\frac{n+1}{2nkN_k}}.
\end{equation}
Also we can introduce a further approximation, since $\det G_k(\hat \mu_k)^{\frac{n+1}{2nkN_k}}\to \delta(\hat E),$ we replace in the above formula $\det G_k(\hat \mu_k)^{\frac{n+1}{2nkN_k}}$ by $\delta(\hat E).$ Finally, we pick $\hat E:=[-1,1]^2$ and $\hat\mu_k$ uniform probability measure on an admissible mesh for the square, for instance the Chebyshev Lobatto grid with $(2k+1)^2$ points, thus our approximation formula becomes
\begin{equation}\label{EquationApproximationFormula}
\delta(E)\approx  \frac 1 2\left(\det \tilde G_k(\mu_k)\frac{1}{\det \tilde G_k(\hat\mu_k)}\right)^{\frac{n+1}{2nkN_k}},
\end{equation}
where we used $\delta([-1,1]^2)=1/2;$ \cite{BlBoLe12}.

Finally, to compute the determinants of the Gram matrices on the right hand side of equation \eqref{EquationApproximationFormula} we use the SVD factorization of the square root of the Gram matrices, note that for instance
$$ \det G_k(\hat\mu_k)=\det\left(\frac 1 {M_k}V_k^H V_k\right)=\left(\det S_k\right)^2=\prod_{j=1}^{N_k}\sigma_j^2,$$
where $V_k$ and $M_k$ has been defined above and $S_k= diag(\sigma_1,\dots,\sigma_{N_k})$ is the diagonal matrix with the singular values the matrix $ V_k/\sqrt{M_k}.$ 

\subsection{Numerical test of the TD-GD algorithm}
In order to illustrate how our algorithm works in practice, we perform two numerical tests for real compact sets whose transfinite diameters have been computed analytically in \cite{BlBoLe12}. Namely,  we consider the case of the unit disk $B^2=\{x\in \R^2: |x|\leq 1\}$ and the unit simplex $S^2:=\{x\in \R_+^2: x_1+x_2\leq 1\}$ For such sets Bos and Levenberg computed formulas that in the specific case of dimension $n=2$ read as
\begin{equation*}
\delta(B^2)=\frac 1{\sqrt{2 e}},\;\;\;\delta(S^2)=\frac 1{2 e}.
\end{equation*}
\subsubsection{TD-GD test case 1: the unit ball in $\R^2$}
\begin{figure}[h]
\caption{The admissible meshes $\hat A_{25}$ (left) and $A_{25}$ (right) of degree $25$ used below for the approximation of the transfinite diameter of the unit disk.}
\label{figmeshes}
\begin{center}
\begin{tabular}{cc}
\includegraphics[scale=0.3]{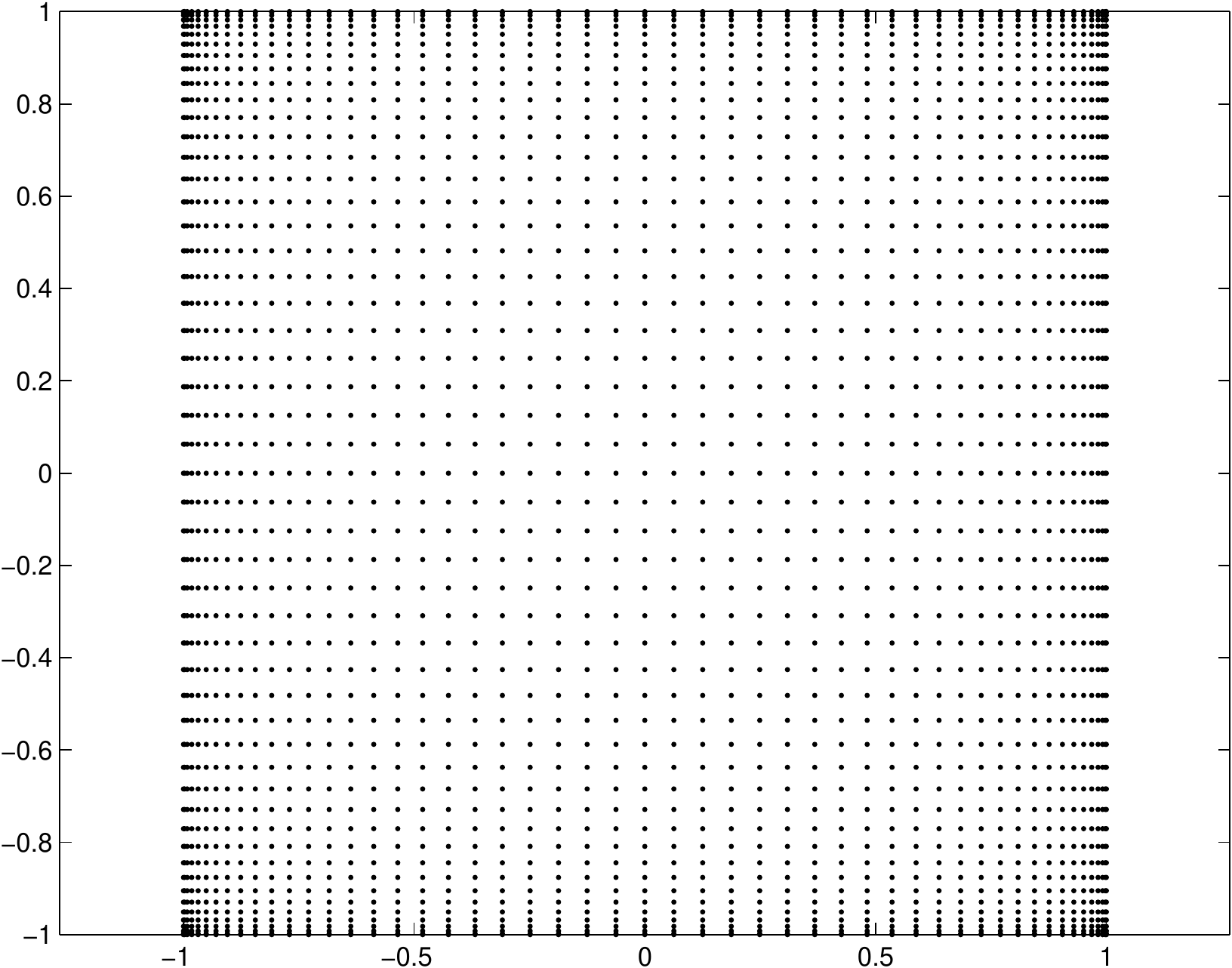} & \includegraphics[scale=0.3]{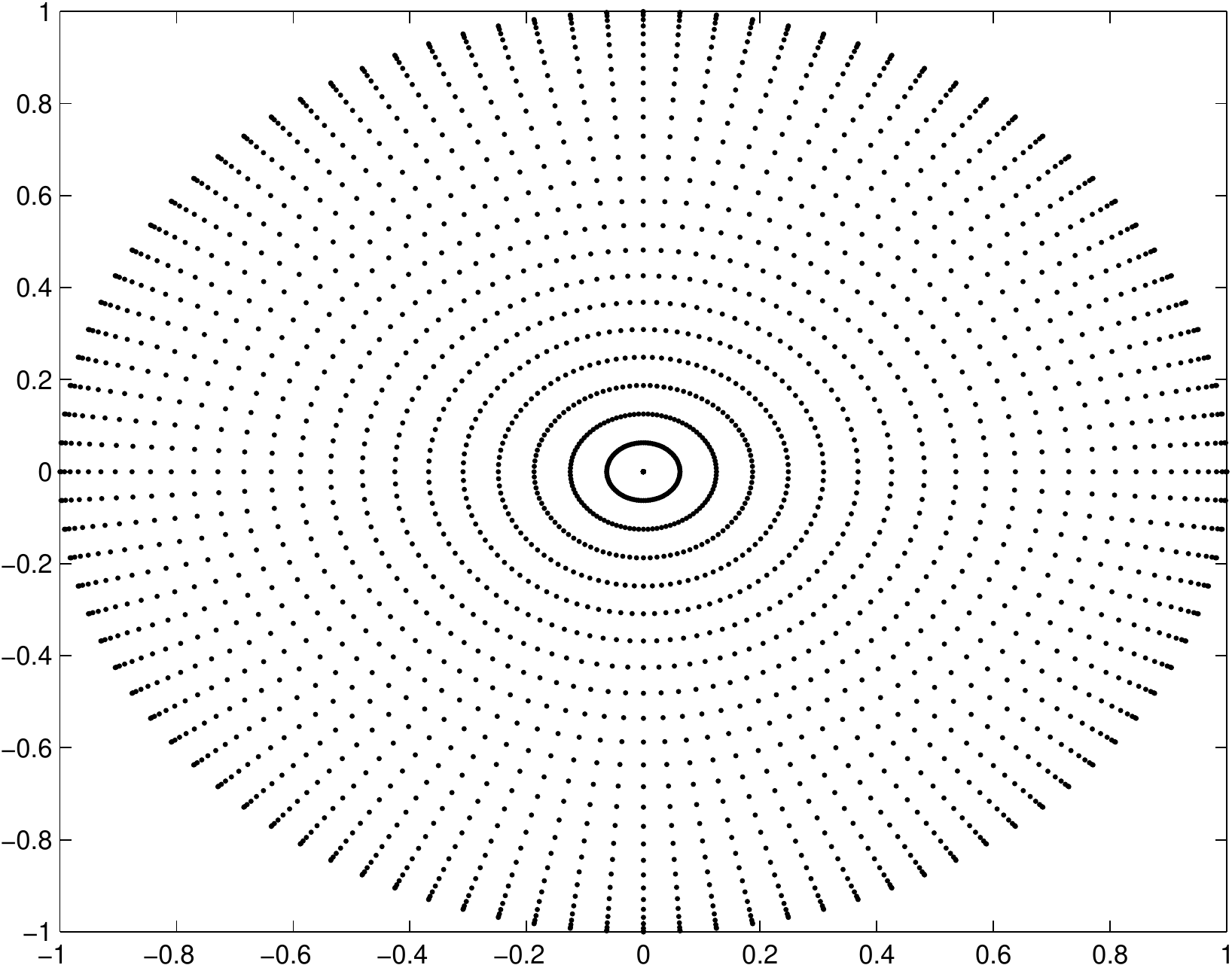}
\end{tabular}
\end{center}
\end{figure}
To compute the approximation of $(G_k(\hat\mu_k))^{\frac{n+1}{2nkN_k}}$ we pick $\hat \mu_k$ to be the probability measure on the point set 
$$\hat A_k:=\{(\cos(i \pi/(2 k)),\cos(j \pi/(2 k))),i,j\in\{0,1,\dots, 2k\}\}.$$
This is a well known admissible mesh of constant $2$ of the square $[-1,1]^2$ for the space of $k$ tensor degree polynomials that in particular include the space $\wp^k$; \cite{EhZe64,BoCaLeSoVi11}.

For $\mu_k$ we use an admissible mesh $A_k$ of degree $k$ built as in \cite{BoCaLeSoVi11} using the radial symmetry of the unit disk. Here 
$$A_k:=\left\{\cos(i \pi/(2 k))(\cos(j \pi/(2 k)),\sin (j \pi/(2 k))),i,j\in\{0,1,\dots, 2k\}\right\}.$$
The admissible meshes $A_k$ and $\hat A_k$ are displayed in Figure \ref{figmeshes}.
\begin{figure}[h]
\begin{center}
\caption{Behaviour of the absolute and relative error of the approximation of the transfinite diameter of the unit disk by the formula \eqref{EquationApproximationFormula} (continuous lines) and by the diagonal of the obtained rho table (dashed lines).}
\label{figtd}
\includegraphics[scale=0.65]{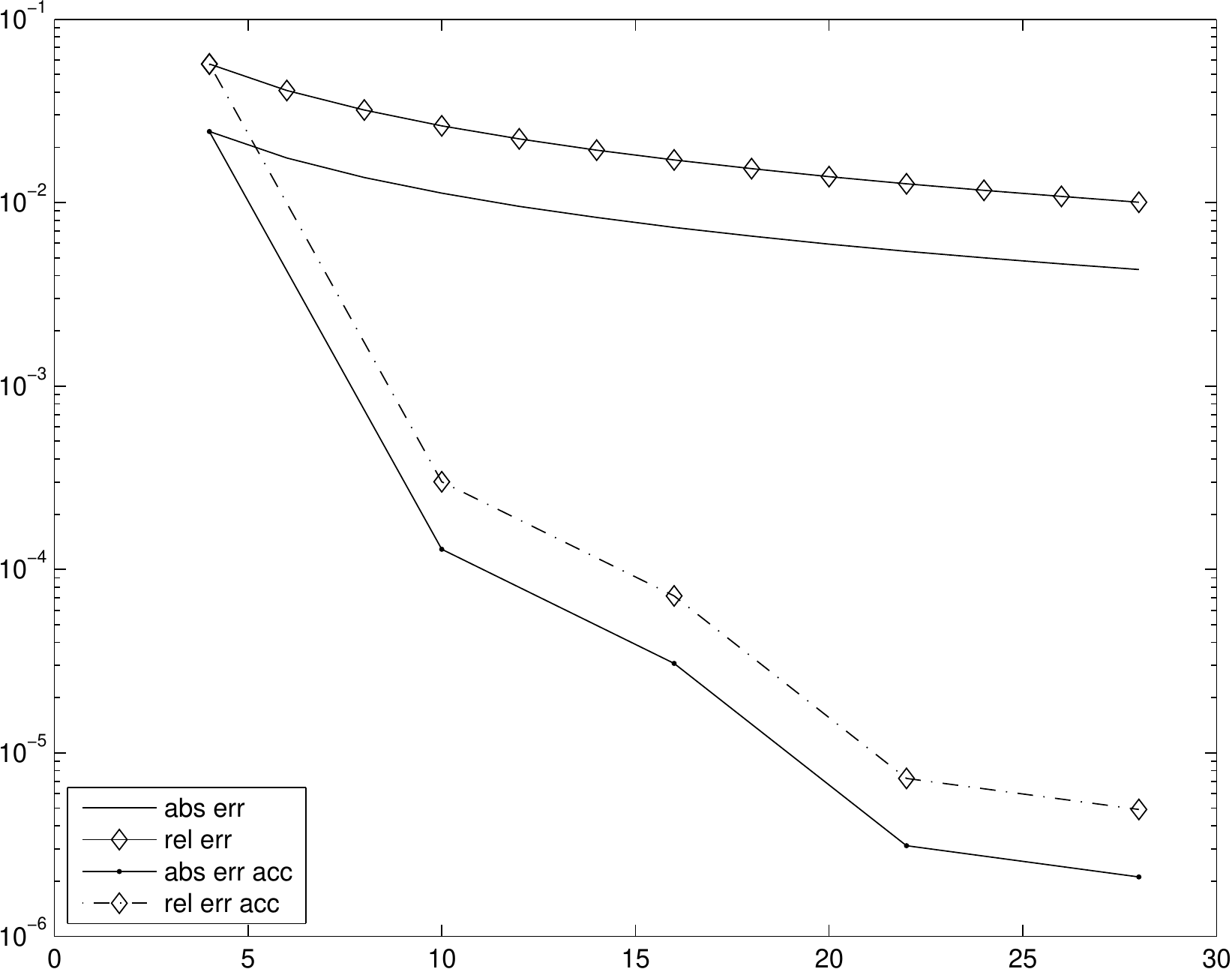}
\end{center}
\end{figure}
We compute the right hand side of equation \eqref{EquationApproximationFormula} for the sequence of values $k=4,6,\dots,28$ and we report both the absolute and the relative errors in Figure \ref{figtd} (continuous line without and with diamonds respectively). On one hand we notice that the convergence rate is very slow, but on the other hand the sequence of approximations is monotone and the error structure is good for the application of the extrapolation. Indeed we report the absolute and relative error (dashed line without and with diamonds respectively) of the sequence obtained by the diagonal of the rho table (rho algorithm) in the same figure. Notice that the absolute error of the accelerated sequence at degree $28$ is $4.9105\cdot 10^{-6}$, that is, \emph{ six digits of $\delta(B^2)$ are computed exactly.}   
\subsubsection{TD-GD test case 2: the unit simplex in $\R^2$}

\begin{figure}[h]
\caption{The admissible meshes of degree $15$ used below for the approximation of the transfinite diameter of the unit simplex.}
\label{figmeshsimplex}
\begin{center}
\includegraphics[scale=0.4]{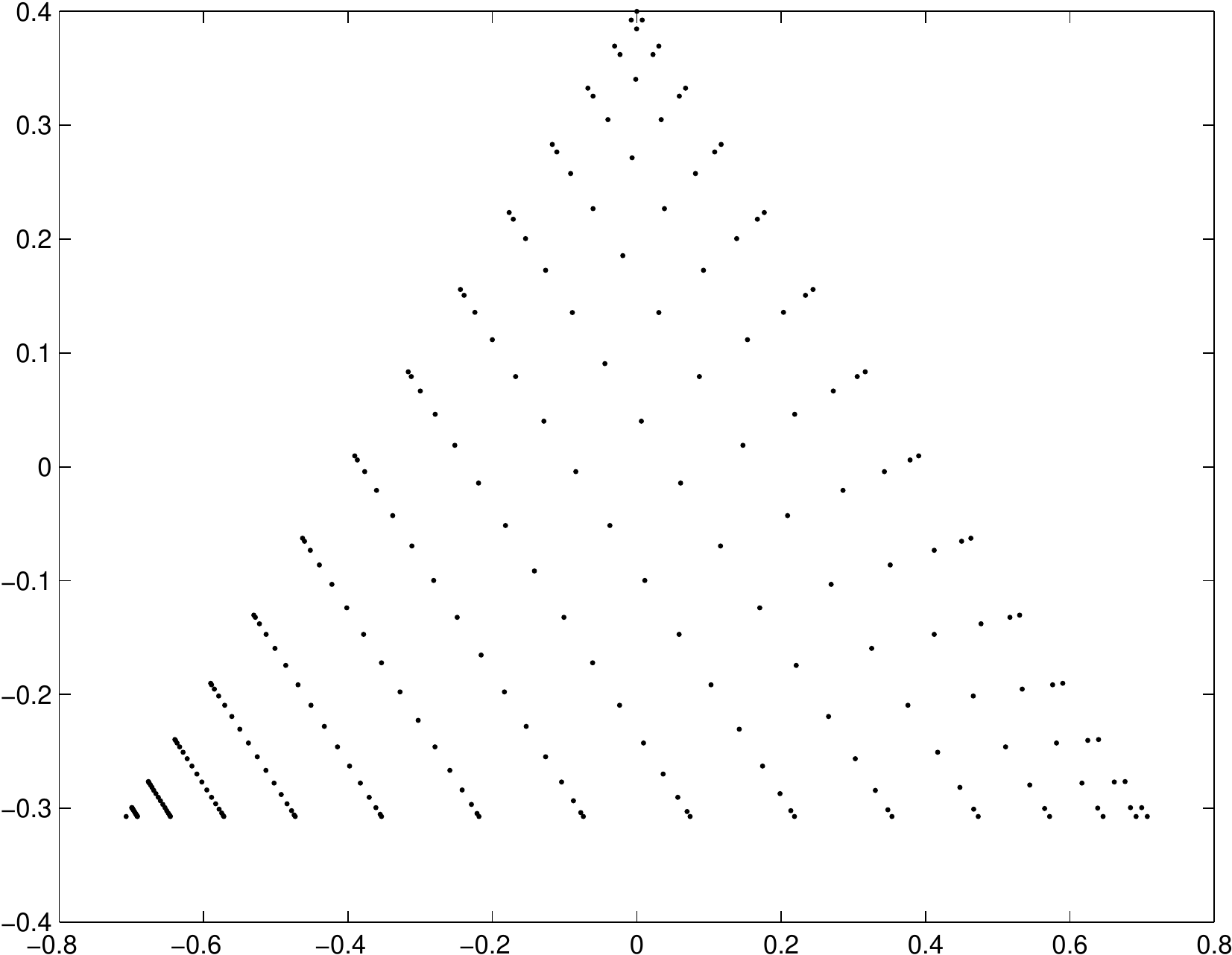}
\end{center}
\end{figure}
The first part of the algorithm for the computation of $\delta(S^2)$, i.e., the computation of the factor in \eqref{EquationApproximationFormula} coming from \eqref{factorchangeofbasis}, is identical to the one we performed for $\delta(B^2).$ 

Then we pick an admissible mesh on $S^2$ following \cite{BoCaLeSoVi11}. Our mesh $A_k$ at the $k$-th stage is the image under the Duffy transformation of a Chebyshev grid on $[-1,1]^2$ formed by $(4k+1)^2$ points, see Figure \ref{figmeshsimplex}. We recall that the Duffy transformation (with a suitable choice of parameters) maps the unit square onto the simplex and any degree $k$ polynomial on the simplex is pulled back by the Duffy transformation  onto the square  to a polynomial of degree not larger than $2k.$ It follows that $\{A_k\}$ is an admissible mesh of constant $2$ for the simplex; \cite{BoCaLeSoVi11}. 

Once the mesh has been defined the numerical computations to get the right hand side of \eqref{EquationApproximationFormula} are performed as above. The obtained results, both in terms of absolute and relative errors, are displayed in Figure \ref{figtdsimplex}. 

Again, the defined algorithm is very slowly converging, nevertheless using extrapolation at infinity by the rho algorithm we get a sequence rather fast converging. Indeed, \emph{more than six exact digits of $\delta(S^2)$ can be computed in less than 10 seconds even on a rather outdated laptop}, e.g., Intel CORE i3-3110M CPU, 4 Gb RAM. 

\begin{figure}[h]
\begin{center}
\caption{Behaviour of the absolute and relative error of the approximation of the transfinite diameter of the unit simplex by the formula \eqref{EquationApproximationFormula} (continuous lines) and by the third column of the obtained rho table (dashed lines).}
\label{figtdsimplex}
\includegraphics[scale=0.65]{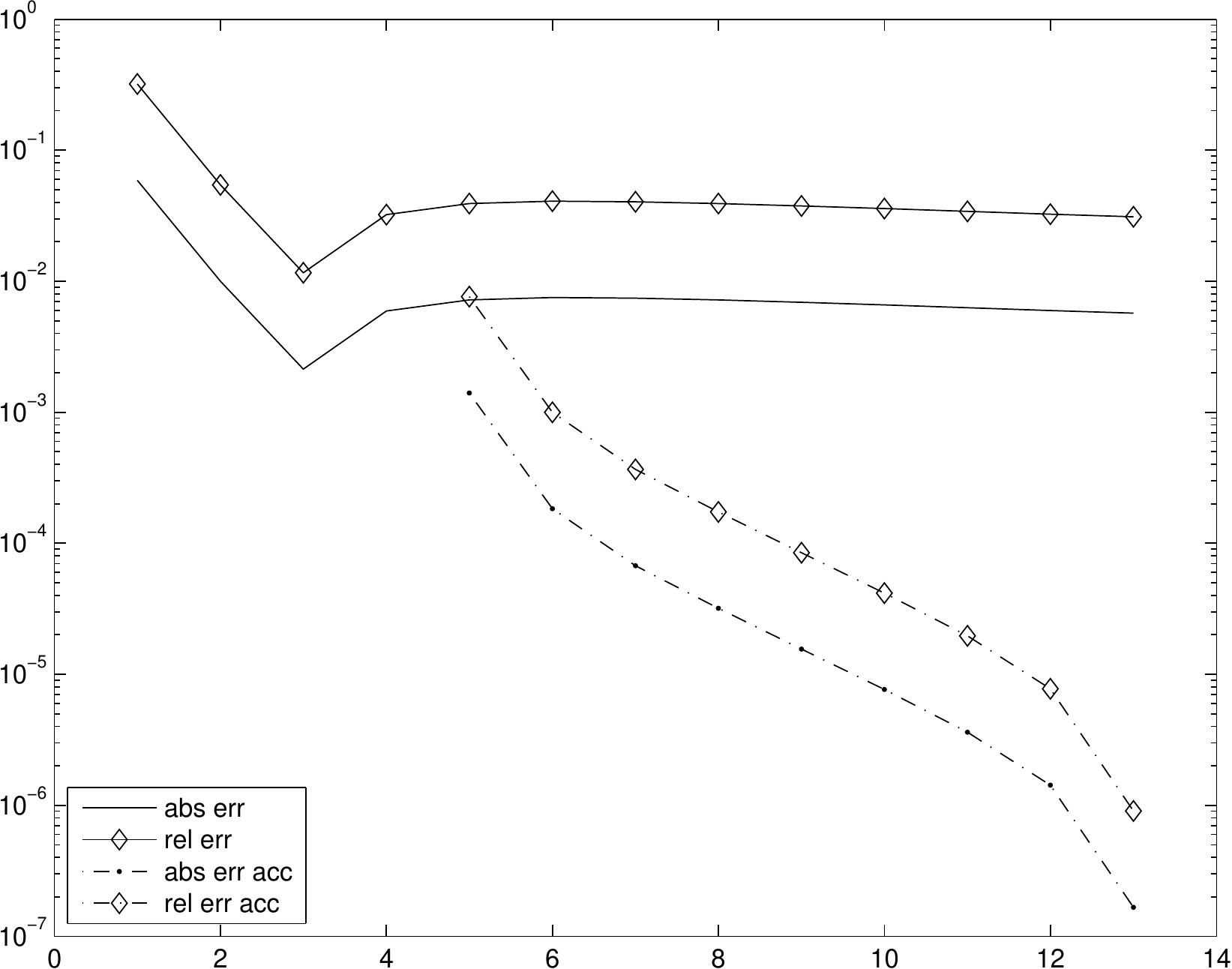}
\end{center}
\end{figure}

\section{Approximating the equilibrium measure}\label{SectionEquilibrium} \label{sectionequilibrium}

Fekete points are (at least theoretically) the first tool  to investigate how the equilibrium measure looks like for a given regular compact set $E\subset \C^n$. Indeed, the main result of \cite{BeBoNy11} asserts that the sequence of uniform probability measures supported at the $k$-th stage on a Fekete array of order $k$ is converging weak$^*$ to the equilibrium measure of the considered set.   Unfortunately Fekete arrays are known analytically for very few instances and they are characterized in general as solutions of an extremely hard optimization problem, hence, though its strong theoretical motivation, this method is not of practical interest.

However, the results in \cite{BeBoNy11} are in fact more general (as shown also in \cite{BlBoLeWa10}): one can take \emph{asymptotically Fekete arrays} (see equation \eqref{asymptoticallyfeketedefinition}) and obtain the same result. This is indeed the approach of \cite[Th. 1]{BoCaLeSoVi11}, where the asymptotically Fekete arrays are produced by a discretizing the optimization problem using an admissible mesh as optimization domain. A $50$-th stage of an asymptotic Fekete array for a regular hexagon is reported in Figure \ref{hexagonfeketeset}.
\begin{figure}[h]
\caption{A degree 50 asymptotic Fekete points set for a regular hexagon computed by the AFP algorithm.}
\label{hexagonfeketeset}
\begin{center}
\includegraphics[scale=0.6]{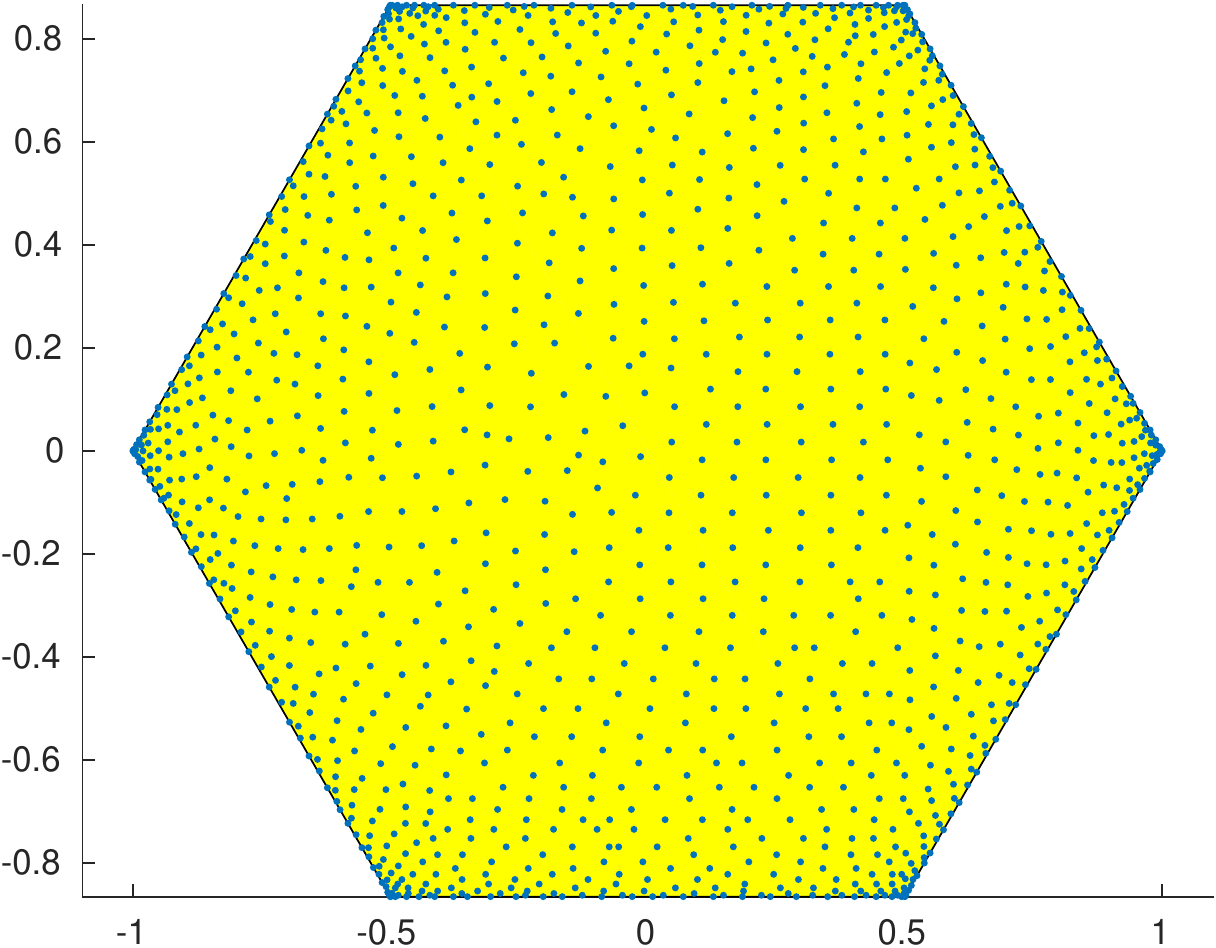}
\end{center}
\end{figure}

Another strategy to get a sequence of (weighted) point masses approaching the equilibrium measure in the weak$^*$ sense is based on the Bergman Asymptotic \eqref{bergmanasymptotic} and the use of admissible meshes. We summarize this in the following proposition, which is a consequence of the results of \cite{BeBoNy11} and \cite{BlBoLeWa10}.
\begin{proposition}\label{bergmanasymptoticonwam}
Let $E\subset \C^n$ a regular compact set. Let $\{A_k\}$ be a weakly admissible mesh for $E$ and  let $\mu_k$ be the uniform probability measure supported on $A_k.$ Denote by $\tilde \mu_k$ the measure $N_k^{-1}B_k^{\mu_k}.$
Then we $\tilde \mu_k$ converges weak$^*$ to $\mu_E.$
\end{proposition}
\begin{proof}[Sketch of the proof]
First we note that the sequence of measures $\{\mu_k\}$ \emph{leads to the transfinite diameter}, i.e. one has the asymptotic result \eqref{equationGramasymptotic} as proven above. This is the starting point for applying the machinery of \cite{BlBoLeWa10}. 

Indeed for such a sequence of measures one still has Lemma 2.7 and Lemma 2.8 of \cite{BlBoLeWa10}, thus using the derivative of the Aubin-Mabuchi energy functional and Lemma 3.1 of \cite{BeBoNy11}, one gets
$$\lim_k \frac{n+1}{nN_k}\int_E u(z) B_k^{\mu_k} d\mu_k(z)=\frac{n+1}{n}\int_E u(z) d\mu_E,\;\;\forall u\in \mathscr C(E).$$   
\end{proof}

Approximations of the equilibrium measure built by means of point masses may have certain interest when one aims to perform approximated computations with equilibrium measure, for instance computing orthogonal series, since the approximation is given in terms of a quadrature rule. On the other hand, such a kind of approximation can not be easily represented to get an insight on how the equilibrium measure looks like for a given $E$; this property becomes relevant if one aims to test or argue conjectures.

In the rest of the section we will introduce an approximation scheme for $\mu_E$ based on absolutely continuous measures with respect to the standard Lebesgue measure.

 Our method is based on the following lemma.
\begin{lemma}
Let $\mu$ be a positive Borel measure. Let us set, for any $k\in \N,$
\begin{align*}
&D=(D_k(z,\mu))_{h,i}:=(\partial_i q_h(z,\mu))\in \C^{N_k\times n}\\
&b=(b_k(z,\mu))_i:=(q_1(z,\mu), \dots q_{N_k}(z,\mu))^T\in \C^{N_k},
\end{align*}
then we have
\begin{equation}\label{lemmasthesis}
\det \left(\partial \bar\partial \frac 1 {2k}\log B_k^\mu(z)\right)=\frac{\det( D^H D)- b^HD\adj(D^H D) D^Hb }{(2 k |b|^2)^n}.
\end{equation}
Here $\adj$ denotes the adjugate of a matrix.
\end{lemma}
\begin{proof}
First notice that $B_k^\mu(z)=\sum_{h=1}^{N_k}q_h(z,\mu)\overline{q_h}(z,\mu)$ is a smooth function never vanishing in $\C^n$, hence we can use classical differentiation with no problems.  

We have
$$\partial \bar\partial \frac 1 {2k}\log B_k^\mu(z)=\frac 1 {2k}\left(\frac{\partial \bar\partial B_k^\mu}{B_k^\mu}-\frac{\bar\partial B_k^\mu}{B_k^\mu}\frac{(\bar B_k^\mu)^T}{B_k^{\mu_k}}\right)=\frac 1{2kB_k^\mu}\left(\partial \bar\partial B_k^\mu+\frac{i\bar\partial B_k^\mu}{\sqrt{B_k^\mu}}\frac{i(\bar\partial B_k^\mu)^T}{\sqrt{B_k^\mu}}\right).$$
Also, using the linearity of differentiation and the tensor structure of $B_k^\mu=b^Hb$ we get
\begin{equation*}
\partial B_k^\mu=D^T \bar b,\;\; \bar\partial B_k^\mu=D^H b,\;\;\partial \bar\partial B_k^\mu=D^HD.
\end{equation*}
So we can write
$$\partial \bar\partial \frac 1 {2k}\log B_k^\mu(z)=\frac 1{2k|b|^2}\left(D^HD+\frac{iD^Hb}{|b|}\frac{i\bar b D}{|b|}   \right).$$
Lastly we use the Matrix Determinant Lemma, i.e., $\det(A+uv^T)=\det A +v^T\adj(A)u,$ and the fact that $\det(\lambda A)=\lambda^n\det A$ to get  equation \eqref{lemmasthesis}.
\end{proof}

We already shown, see Theorem \ref{TheoremAMAsymptotics}, that, for the sequence $\{\mu_k\}$ of uniform probability measures supported on a weakly admissible mesh for $E$, one has the asymptotic
$$\lim_k \frac 1 {2k}\log B_k^{\mu_k}(z)=V_E^*(z)$$
locally uniformly. We recall also that the Monge Ampere operator is continuous under the local uniform limit (see for instance \cite{Kli}), thus 
$$\lim_k \left(\ddc {\frac 1 {2k}\log B_k^{\mu_k}}\right)^n=\lim_k(2i)^n\det\left(\partial\bar \partial\frac 1 {2k}\log B_k^{\mu_k}  \right)\\Vol_{\C^n}=\mu_E,$$
where the limit is to be intended  in the sense of the weak$^*$ topology of Borel measures. Therefore we have the following.  
\begin{theorem}\label{TheoremApproxEqMes}
Let $E\subset \C^n$ a regular compact set. Let $\{A_k\}$ be a weakly admissible mesh for $E$ and  denote by $\mu_k$ the uniform probability measure supported on $A_k.$  Let us denote by $\eta_k$ the sequence of functions
$$\eta_k:=  \det \left(\partial \bar\partial \frac 1 {2k}\log B_k^\mu(z)\right).$$
The sequence $(2 i)^n\eta_k d\Vol_{\C^n}$ converges weak$^*$ to $\mu_E.$
In particular, when $D$ has full rank, we have
\begin{equation}\label{eMformula}
\eta_k:= \frac { \prod_{l=1}^n\sigma_l} {(2k|b|^2)^n} \;\;\frac{b^H}{|b|}\left(\mathbb I_{N_k}-D S^{-1} D^H \right)\frac{b}{|b|},
\end{equation}
where $S=diag(\sigma_1,\dots,\sigma_n)$ is the diagonal matrix $R^HR$ and $D=QR$ is the standard QR factorization of $D.$
\end{theorem}
\begin{proof}
The only thing that remains to prove is equation \eqref{eMformula}. It is sufficient to simply notice that if $A$ is any invertible matrix, then $\adj(A)=(\det(A))A^{-1}.$ In our specific case, in which $S=R^HR$ for a  triangular matrix $R$, $\det S=\prod_{l=1}^n\sigma_l$ factors out and \eqref{eMformula} follows  
\end{proof}
\begin{remark}
Note that the measures $\eta_k$ are not a priori supported on $E$, however it follows trivially by the above theorem that also the sequence of measures having density $\eta_k \chi_E$ (i.e., the restriction to $E$ of $\eta_k$) has the same weak$^*$ limit.
\end{remark}

\section*{Acknowledgements}
The findings of this work are essentially a part of the doctoral dissertation \cite{Pi16T}. Consequently, much of what we present here have been deeply influenced by the discussions with the Advisor Prof. N. Levenberg (Indiana University).
 
All the software used in the numerical tests we performed above has been developed in collaboration with Prof. M. Vianello (University of Padova). The author deeply thanks him both for the scientific collaboration and the support.

\bibliographystyle{abbrv}
\bibliography{biblio}   
\end{document}